\documentclass[11pt,reqno]{amsart}
\usepackage{amsfonts}
\usepackage{amsthm}
\usepackage{amsmath}
\usepackage{amssymb}
\usepackage[dvips]{graphics,graphicx}
\usepackage{caption}
\usepackage[T1]{fontenc}
\usepackage{epsfig}
\usepackage{lmodern}
\usepackage{hyperref}
\usepackage{exscale}
\usepackage[latin1]{inputenc}
\usepackage{xcolor}

\usepackage{fancyhdr}
\usepackage{ifthen}
\usepackage{geometry}
 \selectfont
\theoremstyle{plain}
\newtheorem{thm}{Theorem}[section]
\newtheorem{pro}{Proposition}[section]
\newtheorem{lem}{Lemma}[section]

\newtheorem{res}{Remarks}[section]

\newtheorem{Ex}{Examples}[section]

\def\eps{\varepsilon}
\title[Single-point blow-up for reaction-diffusion systems]{Improved conditions for single-point
blow-up in reaction-diffusion systems}
\author[N. Mahmoudi]{Nejib Mahmoudi}
\address{Universit\'e de Tunis El Manar, Facult\'e des Sciences de Tunis,
 D\'epartement de Math\'ematiques, Laboratoire  \'Equations aux D\'eriv\'ees
Partielles LR03ES04, 2092 Tunis,
Tunisie.}\email{mahmoudinejib@yahoo.fr}
\author[Ph. Souplet]{Philippe Souplet}
\address{Universit\'e Paris 13, Sorbonne Paris Cit\'e, CNRS UMR 7539 LAGA, 99,
Avenue Jean-Baptiste Cl\'ement, 93430 Villetaneuse,
France.}\email{souplet@math.univ-paris13.fr}
\author[S. Tayachi]{Slim Tayachi}
\address{Universit\'e de Tunis El Manar, Facult\'e des Sciences de Tunis, D\'epartement de
Math\'ematiques, Laboratoire  \'Equations aux D\'eriv\'ees
Partielles LR03ES04,  2092 Tunis,
Tunisie.}\email{slim.tayachi@fst.rnu.tn} \subjclass[2010]{Primary:
35B40; 35B44; 35B50. Secondary: 35K61; 35K40; 35K57} \keywords{Nonlinear
initial-boundary value problems, nonlinear parabolic equations,
reaction-diffusion systems, asymptotic behavior   
of solutions, single-point blow-up, blow-up profile.}
\begin{document}
\begin{abstract}  We study positive blowing-up
solutions of the system:
$$u_{t}-\delta\Delta u=v^p,\,\,\, v_{t}-\Delta v=u^{q},$$
as well as of some more general systems.
For any $p,\,q>1$, we prove single-point blow-up for any radially
decreasing, positive and classical solution in a ball. 
This improves on previously known results in 3 directions:

(i) no type~I blow-up assumption is made
(and it is known that this property may fail);

(ii) no equidiffusivity is assumed, i.e. any $\delta>0$ is allowed;

(iii) a large class of nonlinearities $F(u,v)$, $G(u,v)$ can be handled, 
which need not follow a precise power behavior.

As side result, we also obtain lower pointwise estimates for the final blow-up profiles.
\end{abstract}
\maketitle

\section{Introduction}
\subsection{Problem and main results}
 In this paper, we consider nonnegative solutions of the following
reaction-diffusion system:
\begin{equation}\label{1b}
 \left\{
  \begin{array}{ll}
    u_{t}-\delta\Delta u=v^p  &\,x\in\Omega,\,t>0,\\
    v_{t}-\Delta v=u^q,  &\,x\in\Omega,\,t>0,\\
    u=v=0,  &x\in\partial\Omega,\,t>0,\\
    u(0,x)=u_{0}(x), &x\in\Omega,\\
    v(0,x)=v_{0}(x), &x\in\Omega,
\end{array}
\right.
\end{equation}
as well as of the more general system
\begin{equation}\label{1}
 \left\{
  \begin{array}{ll}
    u_{t}-\delta\Delta u=F(u,\,v),  &\,x\in\Omega,\,t>0,\\
    v_{t}-\Delta v=G(u,\,v),  &\,x\in\Omega,\,t>0,\\
    u=v=0,  &x\in\partial\Omega,\,t>0,\\
    u(0,x)=u_{0}(x), &x\in\Omega,\\
    v(0,x)=v_{0}(x), &x\in\Omega.
\end{array}
\right.
\end{equation}
Here $p,\,q>1$, $\delta>0,$ $\Omega=
B(0,\,R)=\left\{x\in\mathbb{R}^n\; ;|x|<R\right\}$ with $R>0$,
\begin{equation}\label{initialdata}
u_{0}, v_{0}\in L^{\infty}(\Omega),
\quad\hbox{ $u_{0}, v_{0}\ge 0$, radially symmetric, radially nonincreasing. }
\end{equation}
As for the functions $F$ and $G$, we assume that
\begin{equation}\label{classical}
F,\,G\in C^1(\mathbb{R}^2)
\end{equation}
and that system (\ref{1}) is cooperative, i.e.:
\begin{equation}\label{monotone}
F_v(u,\,v),\,G_u(u,\,v)\geq0,\quad\hbox{ for all $u, v\ge 0$.} 
\end{equation}
Additional assumptions on $F, G$ will be made below.
 
Under assumptions (\ref{initialdata})--(\ref{monotone}), 
system~(\ref{1}) has a unique nonnegative, radially symmetric and radially nonincreasing
maximal solution $(u,v)$, classical for $t>0$. 
This fact follows by standard contraction mapping and maximum principle arguments. 
The maximal existence time of $(u,v)$ is denoted by $T^\ast\in (0,\infty]$.
If, moreover, $T^\ast<\infty$, then
 \begin{equation*}
\limsup_{t \rightarrow T^{\ast}}\,(\| u(t)\|+\|v(t)\|_{\infty})=\infty,
\end{equation*}
 and we say that the solution blows up in finite time with blow-up
time $T^\ast$.
Also, without risk of confusion, we shall denote $\rho=|x|$, $u=u(t,\rho)$, $v=v(t,\rho)$.
So we have
\begin{equation}\label{monot}
u_\rho,\ v_\rho\le 0 \quad\hbox{ in $(0,T^\ast)\times\overline\Omega$.}
\end{equation}

\smallskip

Problem~(\ref{1b}) is a basic model case for reaction-diffusion systems
and, as such, it has been the subject of intensive investigation for more than 20 years
(see e.g. \cite[Chapter~32]{pavol} and the references therein).
We are here mainly interested in proving
single-point blow-up for systems~(\ref{1b}) and~(\ref{1}). 

\smallskip
For system~(\ref{1b}),
the blow-up set was first studied in \cite{giga}. In that work,
Friedman and Giga proved that blow-up occurs only at
the origin for symmetric nonincreasing initial data in dimension $n=1$,
under the very restrictive conditions $p=q$ and $\delta=1$.  Note that these assumptions are
essential in \cite{giga} in order to apply the maximum principle to
suitable linear combination of the components $u$ and $v$, so as to
derive comparison estimates between them.

Let us recall that, for scalar equations, the first result on single-point blow-up
was obtained by Weissler \cite{fred}, and that different methods
were subsequently developed in \cite{friedman, MullerWeissler}. In turn, the method of
Friedman and Giga for systems is based on an extension of that in
\cite{friedman} for a single equation. More recently, the
restriction $p=q$ was removed by the second author \cite{souplet}, who proved
single-point blow-up for radial nonincreasing solutions of~(\ref{1b}) for any $p,q>1$ and $n\ge 1$. 
However, the equidiffusivity assumption
$\delta=1$ is still needed in  \cite{souplet} and, in addition, it
is required that the solution satisfies the upper type~I blow-up
rate estimates
\begin{equation}\label{7}
\underset{0<t<T^{\ast}}{\sup}(T^\ast-t)^\alpha\|u(t)\|_\infty<\infty,\quad\underset{0<t<T^{\ast}}{\sup}(T^\ast-t)^\beta\|v(t)\|_\infty<\infty,
\end{equation}
where
\begin{eqnarray}\label{62}
\alpha=\frac{p+1}{pq-1},\quad\beta=\frac{q+1}{pq-1}.
\end{eqnarray}

\smallskip
The purpose of this paper, still for any $p,q>1$, is to further
remove the previously made extra assumptions. More precisely, we
shall improve the known results in three directions, by proving
single-point blow-up:

(i) {\bf without assuming the  type~I blow-up} rate estimate (\ref{7});

(ii) {\bf without assuming equidiffusivity}, i.e. for any $\delta>0$;

(iii) including for {\bf general problem} such as (\ref{1}).

Direction (i) seems the more important and challenging one, since
estimate (\ref{7}) is not known in general and need not even be
true. It usually requires either the hypothesis that $p$ or $q$ are
not too large (see e.g. \cite{chelebik, fila}), or that the solution
is monotone in time. Indeed,  for large $p$, even in the particular
case of the scalar problem, there exist radial nonincreasing,
single-point blow-up solutions of type~II (i.e., such that (\ref{7})
fails); see \cite{velazquez2, velazquez3, mizog}.
As for the  case of  monotone in time solutions, it seems that the
known proofs of~(\ref{7}) for systems (see e.g. \cite{deng}) usually require
$\delta=1$. Also we recall that non-equidiffusive parabolic systems
are often much more involved, both in terms of behavior of solutions
and at the technical level (cf.~\cite{pierre} and \cite[Chapter~33]{pavol}). 
As for the general problem (\ref{1}), we shall be able to handle a large class of nonlinearities 
which need not follow a precise power behavior.
The features (i)-(iii) will require a number of nontrivial new ideas,
building on the approach in \cite{souplet}, which is here improved and made more flexible.
See Section~1.2 below for details.

The main results of this paper are the following.

\begin{thm}\label{59a}
Let $\Omega= B(0,\,R)$, $p,\,q>1$ and $\delta >0$.
Assume (\ref{initialdata}) and let the solution $(u,\,v)$ of~(\ref{1b}) satisfy $T^\ast<\infty$.
Then blow-up occurs only at the origin, i.e.
\begin{eqnarray}\label{8}
\underset{0<t<T^{\ast}}{\sup}(u(t,\,\rho)+v(t,\,\rho))<\infty,\quad\hbox{for
all}\,\,\rho\in(0,R).
\end{eqnarray}
\end{thm}

Our next result, which concerns system (\ref{1}),
actually contains Theorem~\ref{59a} as a special case but,
in view of the special interest of system~(\ref{1b}),
we prefered to state Theorem~\ref{59a} separately.
We will assume the following conditions on the functions $F, G$:
\begin{eqnarray}\label{controle}
\hspace{-4cm}& &c_1 v^p\leq F(u,\,v)\leq c_2(v^p+u^r+1),\\\label{controle1}
\hspace{-4cm}& &c_1 u^q \leq G(u,\,v)\leq c_2(u^q+v^s+1),
\end{eqnarray}
for all $u,\,v\geq0$  and for some positive constants  $c_1,\,c_2,$ where
\begin{eqnarray}\label{controle1b}
& &r=\frac{p(q+1)}{p+1}\quad\mbox{and}\quad s=\frac{q(p+1)}{q+1},
\end{eqnarray}
and
\smallskip
\begin{equation}\label{controle2}
 \left\{
  \begin{array}{ll}
& \hbox{for all $C_1,C_2>0$, there exist $\mu,A,\kappa_1,\kappa_2>0$ with $\kappa_1\kappa_2<1$, such that} \\
\noalign{\vskip 2mm}
& \quad 
(1+\mu)F\leq u F_u+ \kappa_1 v F_v\quad \hbox{ and }\quad
(1+\mu)G\leq v G_v+ \kappa_2 u G_u \\
\noalign{\vskip 2mm}
& \hskip 4cm \hbox{ on $D:=\Bigl\{(u,\,v)\in[A,\,\infty)^2\,;\,C_1\leq \displaystyle\frac{u^{q+1}}{v^{p+1}}\leq C_2\Bigr\}$.}
\end{array}
\right.
\end{equation}
\hskip 1pt

\begin{thm}\label{59}
Let $\Omega= B(0,\,R)$, $p,\,q>1,$ $\delta >0$. 
Assume (\ref{initialdata})--(\ref{monotone}) and (\ref{controle})--(\ref{controle2}).
Let the solution $(u,\,v)$ of (\ref{1}) satisfy $T^\ast<\infty$.
Then blow-up occurs only at the origin, i.e. (\ref{8}) holds.
\end{thm}

We immediately give examples of nonlinearities to which Theorem~\ref{59} applies.

\begin{Ex}\label{59bEx}
(i) The result of Theorem~\ref{59} is valid for system (\ref{1}) with 
\begin{equation}\label{defFGex1}
F(u,\,v)=\lambda v^p+\displaystyle{\sum_{i=1}^{m}} \lambda_i u^{r_i} v^{s_i}\quad\mbox{and}\quad G(u,\,v)=\overline{\lambda}u^q+\displaystyle{\sum_{i=1}^{m}} \overline{\lambda}_i u^{\overline{r}_i} v^{\overline{s}_i},
\end{equation}
where $p, q>1$, $m\geq1$ and for all $1\leq i \leq m,$ $r_i,\,s_i,\,\overline{r}_i,\,\overline{s_i}, \lambda_i, \overline{\lambda}_i\geq 0$,
\begin{equation}\label{defFGex2}
r_i\frac{p+1}{q+1}+s_i\leq p\quad\mbox{and}\quad \overline{r}_i+\overline{s}_i\frac{p+1}{q+1}\leq q.
\end{equation}
We note that the requirement that $F, G$ be of class $C^1$ imposes $r_i, s_i, \overline{r}_i, \overline{s}_i\in \{0\}\cup [1,\infty)$.
However, in case some of these numbers belong to $(0,1)$, Theorem~\ref{59} still applies if
$F, G$ only coincide with the expressions in (\ref{defFGex1}) for $u$ and $v$ sufficiently large.
We stress that $F, G$ in (\ref{defFGex1}) are {\it not} mere perturbations of $v^p, u^q$. Indeed, when we have equality in (\ref{defFGex2}),
the additional terms are critical in the sense of scaling. 
\smallskip

(ii) The result of Theorem~\ref{59} is also valid for system (\ref{1}) with
\begin{equation}\label{defFGex3}
F(u,\,v)=v^p\bigl[1+\lambda \sin^2\bigl(k\log(1+v)\bigr)\bigr]\quad\mbox{and}\quad 
G(u,\,v)=u^q\bigl[1+\overline\lambda \sin^2(\overline k \log(1+u)\bigr)\bigr]
\end{equation}
where
\begin{equation}\label{defFGex4}
p, q>1  \quad \lambda, \overline\lambda>0, \quad 0<k<\frac{(p-1)\sqrt{1+\lambda}}{\lambda}
\quad 0<\overline k<\frac{(q-1)\sqrt{1+\overline\lambda}}{\overline\lambda}.\end{equation}
Note that Theorem~\ref{59} thus allows nonlinearities $F, G$ with oscillations of arbitrarily large amplitude
around $v^p, u^q$  (since $\lambda, \overline\lambda$ can be arbitrarily large in (\ref{defFGex4})).
\end{Ex}

Finally, in the case of monotone in time solutions, we extend to
system (\ref{1})  the lower pointwise estimates from \cite{souplet}
on the final blow-up profiles.
\begin{thm}\label{69}
Let $\Omega= B(0,\,R),$ $p,\,q\ge 1$ and $\delta>0$. 
Assume (\ref{initialdata})--(\ref{monotone}), (\ref{controle})--(\ref{controle1b})
and let the solution $(u,\,v)$ of (\ref{1}) satisfy $T^\ast<\infty$. Assume in addition that $u_t, v_t\ge 0$.
Then there exist constants $\eps_0, \eps_1>0$, such that
 \begin{eqnarray*}
& &|x|^{2\alpha}u(T^{\ast},\,x)\geq \eps_0,\qquad 0<|x|<\eps_1
\end{eqnarray*}
and
\begin{eqnarray*}
& &|x|^{2\beta}v(T^{\ast},\,x)\geq \eps_0,\qquad 0<|x|<\eps_1,
\end{eqnarray*}
where $\alpha$ and $\beta$ are given by \eqref{62}.

\end{thm}
\setcounter{re}{1}
\begin{res}\rm{
\item [$(i)$] The results of Theorems \ref{59} and  \ref{69} remain true for the Cauchy
problem (that is, (\ref{1}) with $R=\infty$ and $\partial\Omega=\emptyset$) provided $u_0,$ $v_0$ are
not both constant. These follow from
simple modifications
of the proofs.
\item [$(ii)$] Concerning Theorem \ref{69}, we note that the existence of a positive, radially symmetric, radially
nonincreasing and classical solution of (\ref{1}) such that $T^\ast<\infty$ and $u_t,
v_t\geq0$, can be obtained for initial data $(\lambda u_0,\,\lambda
v_0)$ with $\lambda>0$ large enough, whenever $u_0, v_0$ satisfy~(\ref{initialdata}) and
\begin{equation*}
\left\{
  \begin{array}{ll}
    \mbox{$u_0,\,v_0\in C^2(\Omega)\cap C(\overline\Omega),\quad  u_0=v_0=0$ on $\partial\Omega$}, & \hbox{ } \\
    \noalign{\vskip 1mm}
    \delta\Delta u_0+F(u_0,\,v_0)\geq0, \quad 
    \Delta v_0+G(u_0,\,v_0)\geq0\quad\mbox{ in $\Omega$}. & \hbox{ }
  \end{array}
\right.
\end{equation*}
See \cite{tayachi}.}
\end{res}

\subsection{Outline of proof}
 As in \cite{giga, souplet} (and cf. \cite{friedman, chen}),  the basic idea for proving single-point blow-up
 is to consider auxiliary functions $J,\,\overline{J}$, either of the form (cf.~\cite{giga}):
\begin{equation}\label{choiceJgiga}
J(t,\,\rho)=u_\rho+\eps c(\rho)u^\gamma,\quad
\overline{J}(t,\,\rho)=v_\rho+\eps \overline c(\rho) v^{\overline{\gamma}},
\end{equation}
or (cf.~\cite{souplet}):
\begin{equation}\label{choiceJsouplet}
J(t,\,\rho)=u_\rho+\eps c(\rho)v^\gamma,\quad
\overline{J}(t,\,\rho)=v_\rho+\eps \overline c(\rho) u^{\overline{\gamma}},
\end{equation}
with suitable constants $\gamma, \overline{\gamma}>1$, $\eps>0$ and functions $c(\rho), \overline c(\rho)$. The couple $(J,\,\overline{J})$ satisfies a system of parabolic inequalities
to which one aims at applying the maximum principle, so as to deduce that $J, \overline{J}\leq 0$.
By integrating these inequalities in space, one then obtains
upper bounds on $u$ and $v$ which guarantee single-point blowup at the origin.

However, in the case of systems, such a procedure turns out to require good comparison properties between $u$ and $v$.
Due to the global comparison properties employed in \cite{giga},
the result there for system~(\ref{1b}) imposed the severe restriction $p=q$
(as well as $\delta=1$, because this comparison was shown by applying the maximum principle 
to a linear combination of $u$ and $v$).
For type~I blowup, radially decreasing solutions of~(\ref{1b}) with $\delta=1$ and any $p,q>1$, 
this was overcome in \cite{souplet} by applying a different strategy.
Instead of looking for comparison properties valid everywhere,
one assumed for contradiction that (type~I) single-point blow-up fails and then
established sharp asymptotic estimates near blowup points.
Namely, it was shown that, if $\rho_0>0$ is a blow-up point, then
\begin{equation}\label{estimA0B0}
\lim_{t\to T^\ast}(T^\ast-t)^\alpha u(t,\,\rho)=A_0,\quad \lim_{t\to T^\ast}(T^\ast-t)^\beta u(t,\,\rho)=B_0
\end{equation}
uniformly on compact subsets of $[0,\rho_0)$,
for some uniquely determined constants $A_0, B_0>0$, hence in particular the comparison property
$$\lim_{t\to T^\ast} \Bigl[\frac{u^{p+1}}{v^{q+1}}\Bigr](t,\,\rho)=A_0^{p+1}B_0^{-(q+1)}.$$
These estimates turned out to be sufficient to handle the system satisfied 
by suitable functions of the form $J,\, \overline{J}$ in~(\ref{choiceJsouplet}).
As for estimate (\ref{estimA0B0}), its proof in \cite{souplet} was long and technical, using similarity variables, 
delayed smoothing effects for rescaled solutions, monotonicity arguments and a precise classification
of entire solutions of a related ODE system.

Although we here follow the same basic strategy as in \cite{souplet}, we have been able to make
the method much more flexible, leading to the improvements mentioned above, owing to a number of new ideas,
which we now describe.

(i) An important observation, improving on \cite{souplet}, is that the proof that $J,\, \overline{J}\le 0$ can be reduced to a weaker property than (\ref{estimA0B0}), namely:
\begin{equation}\label{estimC1C2}
\left\{
  \begin{array}{ll}
C_1\le (T^\ast-t)^\alpha u(t,\,\rho)\le C_2\\
\noalign{\vskip 1mm}
C_1\le (T^\ast-t)^\beta v(t,\,\rho)\le C_2\\
      \end{array}
\right.
 \quad\hbox{ in $[T^\ast/2,T^\ast)\times [\rho_1,\,\rho_2]$,}
\end{equation}
for some $0<\rho_1<\rho_2<\rho_0$ and some (unrestricted) constants $C_1, C_2>0$.
Defining $J,\, \overline{J}$ by~(\ref{choiceJgiga}) instead of (\ref{choiceJsouplet}), and localizing the function $c(\rho)$,
this can be achieved by {\it choosing $\gamma,\bar\gamma>1$ suitably close to $1$}
(see Section~2).
  
(ii) Even though the global type~I estimate (\ref{7}) is
unknown in general or may fail, the following local type~I estimate, away from
the origin, can be proved for radially decreasing solutions of the general system~(\ref{1}):
\begin{equation}\label{OutlineUpper}
u(t,\,\rho)\leq C\,\rho^{-n}\,(T^\ast-t)^{-\alpha} \quad\mbox{and}\quad v(t,\,\rho)\leq C\,\rho^{-n}\,(T^\ast-t)^{-\beta}.
\end{equation}
See Proposition~\ref{65}. This is a rather easy consequence of Kaplan's eigenfunction method. This yields in particular
the upper part of the bounds in (\ref{estimC1C2}).

(iii) As for the more delicate lower bounds in (\ref{estimC1C2}), they are proved in three steps. 
The first step (Proposition~\ref{15})  is to establish a nondegeneracy property
which guarantees that $\rho_0\in (0,R)$ is not a blowup point whenever
\begin{equation}\label{OutlineLowerStep1}
(T^\ast-t)^\alpha u(t,\,\rho) \le \eta \quad\hbox{\bf and}\quad (T^\ast-t)^\beta v(t,\,\rho)\le\eta
\end{equation}
at some time $t$ and some $\rho\in(0,\rho_0)$ with $\eta>0$ sufficiently small.
As in \cite{souplet}, the idea is to work in similarity variables and to use delayed smoothing effects,
adapting arguments from~\cite{HVihp, aundrossi}. 
However, a new difficulty arises due to the lack of global type~I upper estimate on $(u,v)$, 
hence of global bound on the rescaled solution.
This is overcome, after truncating the domain, by carefully comparing with a modified solution.
The latter is obtained by
a suitable reflection and supersolution procedure, taking advantage of the local upper bound in (\ref{OutlineUpper})
(see step~1 of the proof of Proposition~\ref{15}).
After passing to similarity variables, the modified solution is now uniformly bounded,
but at the expense of additional terms, generated by the reflection procedure, which appear in the PDE's.
However, these terms can be localized exponentially far away in space for large time,
and thus taken care of in the smoothing effect arguments.


(iv) As a second step in the proof of the lower bounds in (\ref{estimC1C2}),
we prove (see Section~5)  that solutions rescaled around a blow-up point behave, in a suitable sense,
like a continuous distribution solution of the following system of ordinary differential inequalities (ODI):
 \begin{equation}\label{systODI}
 \left\{
  \begin{array}{ll}
    \phi'+\alpha \phi\ge c_1\psi^p, & \hbox{ } \\
    \psi'+\beta \psi\ge c_1\phi^q, & \hbox{ }
      \end{array}
\right.
\end{equation}
on $(-\infty,\infty)$. 
This is proved by a further use of similarity variables, along with the space monotonicity.
Moreover, we single out a simple but crucial property of local interpendence of components 
for such solutions of (\ref{systODI}); namely, $\phi(0)=0$ if and only if $\psi(0)=0$.

(v) Then, as a last step (Section~6), we show that, if {\it one} of the lower bounds in (\ref{estimC1C2}) is violated,
then, owing to point (iv), we have convergence of rescaled solutions to a solution of (\ref{systODI})
such that $\phi(0)=0$ {\it and} $\psi(0)=0$. Restated in terms in $(u,v)$, this leads to
the degeneracy condition (\ref{OutlineLowerStep1}) at some time $t$. But,
in view of point (iii),
this contradicts $\rho_0$ being a blowup point.
\smallskip

We note that, in \cite{souplet}, the study of the particular system~(\ref{1b}) led to the
system of equalities 
\begin{equation}\label{27a}
 \left\{
  \begin{array}{ll}
    \phi'+\alpha \phi= c_1\psi^p, & \hbox{ } \\
    \psi'+\beta \psi= c_1\phi^q, & \hbox{ }
      \end{array}
\right.
\end{equation}
instead of (\ref{systODI}), and a complete classification of entire solutions of~(\ref{27a}) was obtained,
which enabled one to deduce the more precise behavior (\ref{estimA0B0}) at the left of an alleged nonzero blowup point.
We stress that, thanks to the new possibility of arguing through the weaker estimates (\ref{estimC1C2}),
we can now avoid such a classification
(which is not available for the general system~(\ref{systODI})).

\smallskip

The organization of the rest of this paper is as follows. In Section
2, we prove Theorem~\ref{59} (hence Theorem~\ref{59a})
assuming the local upper and lower type~I estimates (\ref{estimC1C2}) near blow-up points.
Sections 3-6 are next devoted to proving these estimates.
In Section~3, we establish upper blowup estimates away from the origin (Proposition \ref{65}). 
In Section~4 we prove the key nondegeneracy property Proposition~\ref{15}.
In Section~5 we show the ODI behavior for rescaled solutions and
the local interpendence of components for the ODI system.
In Section~6 we then prove the lower bounds in (\ref{estimC1C2}) by using a contradiction argument 
and the results of Sections 3-5.
Finally, in Section~7, we establish the
pointwise lower bounds on the blow-up profiles, i.e., Theorem~\ref{69},
and we verify the assertions in Examples~\ref{59bEx}.


\section{Proof of Theorem~\ref{59}
assuming local upper and lower type~I estimates}

\setcounter{equation}{0}

The local upper and lower type~I estimates, in case of existence of nonzero blow-up points,
are formulated in the following proposition.

\begin{pro}\label{30}
Let $\Omega= B(0,\,R),$ $p,\,q>1$, $\delta >0$.
Assume (\ref{initialdata})--(\ref{monotone}) and (\ref{controle})--(\ref{controle1b})
 and let the solution $(u,\,v)$ of (\ref{1}) satisfy $T^\ast<\infty$.
Assume that there exists $\rho_0\in(0,\,R)$ such that
$$\underset{t\rightarrow
T^\ast}{\limsup}\,\big{(}u(t,\,\rho_0)+v(t,\,\rho_0)\big{)}=\infty$$
and let $[\rho_1,\,\rho_2]$ be a compact subinterval of $(0,\,\rho_0).$ Then, there exist 
constants $C_1,\,C_2>0$ (possibly depending on the solution $(u,v)$ and
on $\rho_0, \rho_1, \rho_2$), such that
\begin{equation}\label{310}
C_1\le (T^\ast-t)^\alpha u(t,\,\rho)\le C_2\quad\mbox{on}\quad[T^\ast/2,\,T^\ast)\times[\rho_1,\,\rho_2]
\end{equation}
and
\begin{equation}\label{311}
C_1\le (T^\ast-t)^\beta v(t,\,\rho)\le C_2\quad\mbox{on}\quad[T^\ast/2,\,T^\ast)\times[\rho_1,\,\rho_2].
\end{equation}
In particular, there exist $C_1', C_2'>0$ such that
\begin{equation}\label{31}
C'_1\leq\frac{u^{q+1}(t,\,\rho)}{v^{p+1}(t,\,\rho)}\leq C'_2\quad\mbox{on}\quad[T^\ast/2,\,T^\ast)\times[\rho_1,\,\rho_2].
\end{equation}
\end{pro}

As already explained in Section~1.2, the proof of Proposition~\ref{30} 
will be developed in Sections~3-6, and we shall now prove Theorem~\ref{59} assuming Proposition~\ref{30}.

We introduce the auxiliary $J,\, \overline{J}$ functions defined by
\begin{equation}\label{47}
J(t,\,\rho)=u_\rho+\eps c(\rho)u^\gamma,\quad
\overline{J}(t,\,\rho)=v_\rho+\eps \overline c(\rho)v^{\overline{\gamma}},
\end{equation}
with
\begin{equation}\label{48}
c(\rho)=\sin^2\left(
\frac{\pi(\rho-\rho_1)}{\rho_2-\rho_1}\right),\quad
\overline c(\rho)=\kappa\,c(\rho),\quad \rho_1\leq \rho\leq \rho_2,  
\end{equation}
where $\gamma$, $\overline{\gamma}>1$ and $\eps$, $\kappa$,
$\rho_2>\rho_1>0$ are to be fixed. We note that 
$J$, $\overline{J}$ $\in C((0,\,T^\ast)\times[0,\,R])\cap
W^{1,2;k}_{loc}((0,\,T^{\ast})\times[0,\,R))$, for all $1<k<\infty$,
 by parabolic $L^p$-regularity.

\begin{lem}\label{49}
 Under the hypotheses of Theorem~\ref{59},
assume that there exists $\rho_0\in(0,\,R)$ such that
$$\underset{t\rightarrow
T^\ast}{\limsup}\,\big{(}u(t,\,\rho_0)+v(t,\,\rho_0)\big{)}=\infty$$
and let $\rho_1=\rho_0/4$ and $\rho_2=\rho_0/2$.
Then there exist $\gamma$,
$\overline{\gamma}>1$, $\kappa>0$ and $T_1 \in(0,\,T^\ast)$, such that,
for any $\eps\in(0,\,1]$,
 the functions $J$ and $\overline{J}$ defined in (\ref{47})--(\ref{48})
satisfy
\begin{equation}\label{50}
    \hspace{-0,5cm}\left\{
      \begin{array}{lll}
           \hspace{-0,2cm}\hfill J_t-\delta J_{\rho\rho}-\delta\displaystyle\frac{n-1}{\rho}J_\rho+\delta\frac{n-1}{\rho^2}J&\leq F_v(u,\,v)\overline{J}
 +\Big{[}F_u(u,\,v)-2\eps\delta\gamma c'u^{\gamma-1}\Big{]}J, & \hbox{ }\\
      \hspace{-0,2cm}\hfill \overline{J}_t-\overline{J}_{\rho\rho}-\displaystyle\frac{n-1}{\rho}\overline{J}_\rho+\frac{n-1}{\rho^2}\overline{J}
        &\leq G_u(u,\,v)J+\Big{[}G_v(u,\,v)-2\eps\overline{\gamma} \,\overline c'v^{\overline{\gamma}-1}\Big{]}\overline{J}, &\hbox{ }
      \end{array}
    \right.
\end{equation}
for a.e. $(t,x)\in[T_1,\,T^\ast)\times(\rho_1,\, \rho_2 ).$
\end{lem}
\begin{proof} 
{\bf Step 1.} {\it Computation of a parabolic operator on $J$ and $\bar J$.}
\smallskip

Let $H= u^\gamma$. By differentiation of (\ref{47}), we have
\begin{eqnarray*}
 J_t-\delta J_{\rho\rho}\hspace{-0,6cm}& &=(u_\rho)_t +\eps c H_t-\delta(u_{\rho\rho})_\rho-\delta\eps c''H-2\delta\eps c' H_\rho-\delta\eps c H_{\rho\rho}\\
 \hspace{-0,6cm}& &=(u_t-\delta u_{\rho\rho})_\rho+\eps\Big{(}c\big{(}H_t-\delta H_{\rho\rho}\big{)}-2\delta c'H_\rho-\delta c''H\Big{)}.
\end{eqnarray*}
 By the first equation in (\ref{1}), we get
\begin{eqnarray*}
 (u_t-\delta u_{\rho\rho})_\rho=\left(\delta\frac{n-1}{\rho}u_\rho+F(u,\,v)\right)_\rho=\delta\frac{n-1}{\rho}u_{\rho\rho}-\delta\frac{n-1}{\rho^2}u_\rho+\,F_u u_\rho+F_v v_\rho
\end{eqnarray*}
 and
\begin{eqnarray*}
H_t-\delta H_{\rho\rho}\hspace{-0,6cm}& &=\gamma u^{\gamma-1}u_t-\delta\gamma(\gamma-1)u^{\gamma-2} u^2 _\rho-\delta\gamma u^{\gamma-1}u_{\rho\rho}\\
\hspace{-0,6cm}& &\leq\gamma u^{\gamma-1}\big{(}u_t-\delta u_{\rho\rho}\big{)}
 =\gamma u^{\gamma-1}\left(\delta\frac{n-1}{\rho}u_\rho +\,F \right).
\end{eqnarray*}
Here and in the sequel, we omit the arguments $u,v$ when no confusion may arise.
 Using this, along with
$u_\rho=J-\eps c u^\gamma$ and
$v_\rho=\overline{J}-\eps \overline c v^{\overline{\gamma}}$, we obtain
\begin{eqnarray*}
J_t-\delta J_{\rho\rho}\hspace{-0,6cm}& &\leq\delta\frac{n-1}{\rho}\left(J-\eps cu^\gamma\right)_\rho-\delta\frac{n-1}{\rho^2}\left(J-\eps cu^\gamma\right)\\
 \hspace{-0,6cm}& &\hspace{0,3cm}+\,F_u \left(J-\eps c u^\gamma\right)+F_v \left(\overline{J}-\eps \overline c v^{\overline{\gamma}}\right)\\
 \hspace{-0,6cm}& &\hspace{0,3cm}+\eps u^{\gamma-1}\left[\gamma c\left(\delta\frac{n-1}{\rho}u_\rho +F \right)-2 \gamma\delta c'u_\rho-\delta c''u\right]\\
 \hspace{-0,6cm}& &=\delta\frac{n-1}{\rho}J_\rho-\delta\eps\frac{n-1}{\rho}c'u^\gamma-\delta\eps c\frac{n-1}{\rho}\gamma u^{\gamma-1}u_\rho-\delta\frac{n-1}{\rho^2}J\\
 \hspace{-0,6cm}& &\hspace{0,3cm}+\,\delta \eps \frac{n-1}{\rho^2}c u^\gamma +F_u \left(J-\eps c u^\gamma\right)+F_v \left(\overline{J}-\eps \overline c v^{\overline{\gamma}}\right)\\
 \hspace{-0,6cm}& &\hspace{0,3cm}+\,\eps u^{\gamma-1}\left[\gamma c\left(\delta\frac{n-1}{\rho}u_\rho+F \right)-2 \delta\gamma c'\left(J-\eps c u^{\gamma}\right)-\delta c''u\right].
\end{eqnarray*}
Consequently,
\begin{equation}
 J_t-\delta J_{\rho\rho}-\delta\frac{n-1}{\rho}J_\rho+\delta\frac{n-1}{\rho^2}J\leq F_v \overline{J}
 +\Big{[}F_u -2\eps\delta\gamma c'u^{\gamma-1}\Big{]}J+\eps H_1,\quad \label{52}
\end{equation}
with
$$
 H_1:= - c u^\gamma F_u -\overline c v^{\overline{\gamma}}F_v
 +\,u^{\gamma-1}\left[\gamma cF +2\delta\eps\gamma c'c u^\gamma+\delta u\left(\frac{n-1}{\rho}\Big{(}\frac{c}{\rho}-c'\Big{)}-c''\right)\right].
$$
 For convenience, we set
$$\xi(\rho)=\frac{n-1}{\rho}\Big{(}\frac{1}{\rho}-\frac{c'}{c}\Big{)}-\frac{c''}{c},
\qquad \rho_1<\rho<\rho_2$$
and, on $(0,T^\ast)\times (\rho_1,\rho_2)$,
\begin{equation}
\widetilde{H}_1:=\frac{H_{1}}{c\,u^{\gamma-1}}=-uF_u-\kappa\frac{v^{\overline{\gamma}-1}}{u^{\gamma-1}}vF_v \\
+\gamma F +2\delta\,\eps \gamma c'u^\gamma+ \delta\xi(\rho)u.
\label{53}
\end{equation}
 Note that, up to now, our calculations made use of (\ref{1})
through the first PDE only.
Thus, by replacing $\delta$ with $1$ and exchanging the roles of $u,$ 
$F, \gamma, c$ and $v$, $G, \overline{\gamma}, \overline c$, we get
\begin{multline}
\overline{J}_t-\overline{J}_{\rho\rho}-\frac{n-1}{\rho}\overline{J}_\rho+\frac{n-1}{\rho^2}\overline{J}\leq G_u J+\Big{[}G_v -2\eps\overline{\gamma}\, \overline c'v^{\overline{\gamma}-1}\Big{]}\overline{J}+\eps H_2,
\end{multline}
with
\begin{equation}
\widetilde{H}_{2}:=\frac{H_{2}}{\overline c\,v^{\overline{\gamma}-1}}:=-
vG_v-\frac{1}{\kappa}\frac{u^{\gamma-1}}{v^{\overline{\gamma}-1}}uG_u
+\overline{\gamma}G+2\,\eps \overline{\gamma} \,\overline c'v^{\overline{\gamma}}
+\xi(\rho)v.
\label{55}
\end{equation}

Next setting $\ell=\rho_2-\rho_1=\rho_0/4$, we have
$$-\frac{c'}{c}
 =-\frac{2\pi}{\ell}\cot\Bigl(\frac{\pi(\rho-\rho_1)}{\ell}\Bigr)
\quad\hbox{ and }\quad
 -\frac{c''}{c}=-\frac{2\pi^{2}}{\ell^2}\cot^2\Bigl(\frac{\pi(\rho-\rho_1)}{\ell}\Bigr)+\frac{2\pi^2}{\ell^2}\\
$$
 hence,
\begin{eqnarray*}
\xi(\rho)
=\frac{n-1}{\rho^2}+\frac{2\pi^2}{\ell^2}
- \frac{2\pi}{\ell}\left[\frac{n-1}{\rho}+\frac{\pi}{\ell}
\cot\Bigl(\frac{\pi(\rho-\rho_1)}{\ell}\Bigr)\right]\cot\Bigl(\frac{\pi(\rho-\rho_1)}{\ell}\Bigr).
\end{eqnarray*}
 It follows that
\begin{equation*}
 \xi(\rho)
\ \underset{\rho\rightarrow \rho_1^+}{\longrightarrow} -\infty
\quad\hbox{ and }\quad
\xi(\rho)
\ \underset{\rho\rightarrow \rho_2^-}{\longrightarrow} -\infty.
\end{equation*}
Since  $\xi$
 is continuous on $(\rho_1,\,\rho_2)$, then there exists $C_3=C_3(n,\,\rho_0)>0$ such that
\begin{eqnarray}\label{58}
\xi(\rho)\leq
C_3,\quad\hbox{for all}\,\,\rho\in(\rho_1,\,\rho_2).
\end{eqnarray}
By (\ref{53}), (\ref{55}) and (\ref{58}), we obtain,
for some $C_4=C_4(\delta,\,\rho_0)>0$,
\begin{align}\label{55b}
\widetilde{H}_1\leq -uF_u-\kappa\frac{v^{\overline{\gamma}-1}}{u^{\gamma-1}}vF_v 
+\gamma F+C_4\delta\gamma u^\gamma+ \delta C_3 u
\end{align}
and
\begin{align}\label{55c}
\widetilde{H}_2\leq-vG_v-\frac{1}{\kappa}\frac{u^{\gamma-1}}{v^{\overline{\gamma}-1}}uG_u 
+\overline{\gamma} G
 +C_4\overline\gamma v^{\overline\gamma}+C_3 v.
\end{align}

{\bf Step 2.} {\it Estimation of the remainder terms $\widetilde{H}_{1}, \widetilde{H}_{2}$
with help of the local lower and upper type~I estimates.}

\smallskip

Assume that $\gamma$ satisfies
\begin{eqnarray}\label{56}
& &1<\gamma<p\frac{q+1}{p+1}
\end{eqnarray}
and set
\begin{eqnarray}\label{78}
 & &\overline{\gamma}=1+\frac{p+1}{q+1}(\gamma-1)
\end{eqnarray}
which, in turn, guarantees
\begin{eqnarray}\label{561}
& &1<\overline{\gamma}<q\frac{p+1}{q+1}.
\end{eqnarray}
Let the constants $C_1,\,C_2>0$ be given by Proposition~\ref{30}. 
By (\ref{310})-(\ref{311}), (\ref{78}) and (\ref{62}), we then have
\begin{equation}\label{781}
\begin{array}{ll}
&\Bigl(C_1C_2^{-\frac{p+1}{q+1}}\Bigr)^{\gamma-1} =\displaystyle\frac{ C_1^{\gamma-1}}{C_2^{\bar\gamma-1}}\leq\frac{u^{\gamma-1}}{v^{\overline{\gamma}-1}}\leq 
  \frac{ C_2^{\gamma-1}}{C_1^{\bar\gamma-1}}
  =\Bigl(C_2C_1^{-\frac{p+1}{q+1}}\Bigr)^{\gamma-1}\\
  & \qquad\qquad\qquad\qquad\qquad\qquad\qquad\qquad\qquad \mbox{on}\:[T^\ast/2,\,T^\ast)\times(\rho_1,\,\rho_2).
\end{array}
\end{equation}
Next, by (\ref{310})-(\ref{31}) and assumption (\ref{controle2})
(with $C'_1, C'_2$ in place of $C_1, C_2$), there exist $\kappa_1, \kappa_2, \mu>0$ 
with $\kappa_1\kappa_2<1$ and $T_0\in (T^\ast/2,T^\ast)$, such that 
\begin{eqnarray}\label{uFueta1}
uF_u+\kappa_1vF_v\ge (1+2\mu)F 
\quad\mbox{on}\:[T_0,\,T^\ast)\times(\rho_1,\,\rho_2)
\end{eqnarray}
and
\begin{eqnarray}\label{uFueta2}
vG_v+\kappa_2uG_u\ge (1+2\mu)G
\quad\mbox{on}\:[T_0,\,T^\ast)\times(\rho_1,\,\rho_2).
\end{eqnarray}
Choose $\kappa$ in (\ref{48}) such that $\kappa_1<\kappa<1/\kappa_2$.
Then taking $\gamma>1$ close enough to~$1$, we deduce from (\ref{781}) that
\begin{eqnarray}\label{KK1K2}
\kappa\frac{v^{\overline{\gamma}-1}}{u^{\gamma-1}}\ge \kappa_1
\quad \mbox{and}\quad \frac{1}{\kappa}\frac{u^{\gamma-1}}{v^{\overline{\gamma}-1}}\ge \kappa_2
\quad\mbox{on}\:[T_0,\,T^\ast)\times(\rho_1,\,\rho_2),
\end{eqnarray}
and we may also assume that 
\begin{eqnarray}\label{gammaeta}
\gamma\le 1+\mu,\quad \bar\gamma\le 1+\mu
\end{eqnarray}
and that (\ref{56}), (\ref{561}) are satisfied.
On the other hand, since $F\ge c_1v^p$ and $G\ge c_1u^q$, it follows from (\ref{31}), (\ref{56}) and (\ref{561})
that there exists $T_1\in (T_0,T^*)$ such that
\begin{eqnarray}\label{etaF}
C_4\delta \gamma u^\gamma+C_3\delta u\le Cv^{\frac{p+1}{q+1}\gamma}\le \mu F
\quad\mbox{on}\:[T_1,\,T^\ast)\times(\rho_1,\,\rho_2)
\end{eqnarray}
and
\begin{eqnarray}\label{etaG}
C_4 \bar\gamma u^{\bar\gamma}+C_3 v\le Cu^{\frac{q+1}{p+1}\bar\gamma}\le \mu G
\quad\mbox{on}\:[T_1,\,T^\ast)\times(\rho_1,\,\rho_2).
\end{eqnarray}

Combining (\ref{55b}), (\ref{55c}) with (\ref{78})-(\ref{etaG}) and using $F_v, G_u\ge 0$, we deduce that 
$$\widetilde{H}_1\leq -uF_u-\kappa_1vF_v +(1+2\mu) F\le 0
\quad\mbox{on}\:[T_1,\,T^\ast)\times(\rho_1,\,\rho_2)$$
and
$$\widetilde{H}_2\leq -vG_v-\kappa_2uG_u +(1+2\mu) G\le 0
\quad\mbox{on}\:[T_1,\,T^\ast)\times(\rho_1,\,\rho_2)$$
and the Lemma follows from (\ref{52})--(\ref{55}).
\end{proof}

With Proposition~\ref{30} and Lemma \ref{49} at hand, we can now conclude the proof of Theorem~\ref{59}.
\smallskip

\begin{proof}[Proof of Theorem~\ref{59}.] Let $(u,\,v)$ be a solution of
system (\ref{1}) satisfying the hypotheses of Theorem~\ref{59} and assume for contradiction that there exists
$\rho_0\in(0,\,R)$ such that
\begin{eqnarray}\label{60}
\underset{t\rightarrow
T^\ast}{\limsup}\,(u(t,\,\rho_0)+v(t,\,\rho_0))=\infty.
\end{eqnarray}
 Also, since $(u,\,v)\not\equiv (0,\,0)$, it is easy to see that $u,\,v>0$
in $(0,\,T^\ast)\times [0,\,R)$, hence $u_\rho(t,\,\cdot)\not\equiv 0$
and $v_\rho(t,\,\cdot)\not\equiv 0$ for each $t\in (0,\,T^\ast)$.
Next, we have
$u_t-\delta u_{\rho\rho}-\delta\frac{n-1}{\rho}u_\rho=f(t,\,\rho)$
on $(0,\,T^\ast)\times(0,\,R),$ with $f(t,\,\rho)=F(u,\,v)$. Since, $u_\rho,\, v_\rho\leq0$ and $F_v\geq0$,
a strong maximum principle
(which can be seen from straightforward modifications of the proof of \cite[Lemma 52.18, p. 519]{pavol}) then guarantees
\begin{equation}\label{urho1}
u_\rho<0\quad\hbox{ on $(0,\,T^\ast)\times(0,\,R)$,}
\end{equation}
and similarly
\begin{equation}\label{urho2}
v_\rho<0\quad\hbox{ on $(0,\,T^\ast)\times(0,\,R)$.}
\end{equation}

Set $\rho_1=\rho_0/4,$ $\rho_2=\rho_0/2$ and let $J,$ $\overline{J}$, $T_1$ be given by  Lemma \ref{49}.
 Since $c(\rho_1)=c(\rho_2)=0$, we have
$J,$ $\overline{J}\leq0$ on
$\bigl((T_1,\,T^\ast)\times\{\rho_1\}\bigr)\cup\bigl((T_1,\,T^\ast)\times\{\rho_2\}\bigr)$.
Taking $\eps> 0$ sufficiently small and using (\ref{urho1}), (\ref{urho2}),
we see that $J,$ $\overline{J}\leq0$ on
$\{T_1\}\times[\rho_1,\,\rho_2]$.
Then, owing to assumption (\ref{monotone}), we may use the maximum principle (as in, e.g., \cite{souplet}), 
to obtain $J,$ $\overline{J}\leq0$ on
$(T_1,\,T^\ast)\times[\rho_1,\,\rho_2].$
Consequently,
\begin{eqnarray*}
-u_\rho\hspace{-0,6cm}& &\geq\eps c(\rho)\, u^\gamma
\quad\hbox{ on $(T_1,\,T^\ast)\times[\rho_1,\,\rho_2].$}
\end{eqnarray*}
 By integration, we obtain
\begin{eqnarray*}
 & &u^{1-{\gamma}}(t,\,\rho_2)\geq (\gamma-1)\eps
 \int_{\rho_1}^{\rho_2}\sin^2\left( \frac{\pi(\rho-\rho_1)}{\rho_2-\rho_1}\right)d\rho>0\quad\hbox{for all } T_1\leq t<T^\ast.
\end{eqnarray*}
It follows that $u(t,\,\rho_2)$ is bounded for $ T_1\leq t<T^\ast$,
and similarly $v(t,\,\rho_2)$ is bounded for $ T_1\leq t<T^\ast$.
Since $u_\rho,$ $v_\rho\leq0$, this leads to a contradiction
with (\ref{60}) and proves the theorem.
\end{proof}

\section{Upper type~I estimates away from the origin}

\setcounter{equation}{0} 

\begin{pro}\label{65}
 Let $\Omega= B(0,\,R),$ $p,\,q>1$, $\delta >0$.
 Assume that (\ref{initialdata})--(\ref{monotone}) are satisfied and that, for some~$c_1>0$,
\begin{equation}\label{28a}
F(u,v)\ge c_1v^p,\quad G(u,v)\ge c_1u^q,\quad\hbox{for all $u,v\ge 0$.}
\end{equation}
Let the solution $(u,\,v)$ of (\ref{1}) satisfy $T^\ast<\infty$.
Then, there exists a constant $M_0>0$
(depending only on $n,\,p,\,q,\,\delta,\,c_1,\,R,\,T^\ast)$ such that
\begin{equation}\label{28}
 u(t,\,\rho)\leq M_0\,\rho^{-n}\,(T^\ast-t)^{-\alpha} \quad\mbox{and}\quad v(t,\,\rho)\leq M_0\,\rho^{-n}\,(T^\ast-t)^{-\beta},
\end{equation}
for  all $t\in[0,\,T^\ast)$ and $0<\rho\leq R.$
\end{pro}

The argument, which is based on Kaplan's eigenfunction method,
 is well known for scalar equations (see e.g. \cite{matanomerle}
and  \cite[Propositions~4.4, 4.6 and Corollary 4.5, pp. 895-896]{MullerWeissler})
and can be easily adapted to systems.

\begin{proof}
We denote by $\lambda_{1}$ the
first eigenvalue of $-\Delta$ in $H_{0}^{1}(B(0,\,R))$ and
$\varphi_{1}$ the corresponding eigenfunction such that
$\varphi_{1}>0$ and $\int_{B(0,\,R)}\varphi_{1}(x)dx=1$.
Multiplying (\ref{1}) by $\varphi_1$, using (\ref{28a}) and integrating by parts, we obtain, on $(0,T^\ast)$,
\begin{eqnarray*}
 \frac{d}{dt}\int_{B(0,\,R)}u(t,\,x)\varphi_1(x) dx\hspace{-0,6cm}& &\ge c_1\int_{B(0,\,R)}v^p(t,\,x)\varphi_1(x)dx 
 -\delta\lambda_1\int_{B(0,\,R)}u(t,\,x)\varphi_1(x)dx,\\
 \frac{d}{dt}\int_{B(0,\,R)}v(t,\,x)\varphi_1(x) dx\hspace{-0,6cm}& &\ge c_1\int_{B(0,\,R)}u^q(t,\,x)\varphi_1(x)dx 
 -\lambda_1\int_{B(0,\,R)}v(t,\,x)\varphi_1(x)dx.
  \end{eqnarray*}
Let $y(t)=\int_{B(0,\,R)}u(t,\,x)\varphi_1(x)dx$ and
$z(t)=\int_{B(0,\,R)}v(t,\,x)\varphi_1(x)dx$.
By Jensen's inequality, we deduce that
$$y'(t)\geq c_1 z^p(t)-\delta\lambda_1 y(t),\qquad z'(t)\geq c_1 y^q(t)-\lambda_1 z(t).$$
We put $Y(t)=e^{\delta\lambda_1 t}y(t)$ and $Z(t)=e^{\lambda_1
t}z(t).$ Then, there exists $C>0$ such that
$$Y'(t)\geq C Z^p(t),\qquad Z'(t)\geq C Y^q(t)\qquad\hbox{on $(0,T^\ast)$.}$$
Here and in the rest of the proof, $C$ denotes a
positive constant depending only on $T^\ast,\,\delta,\,p,\,q,\,n,\, R$ and which may vary from line to line.
By \cite[Lemma 32.10, p. 284]{souplet}, there exists $C$ such that
$$Y(t)\leq C (T^\ast-t)^{-\alpha},\qquad Z(t)\leq C (T^\ast-t)^{-\beta}\qquad\hbox{ on $[0,T^\ast)$,}$$
where $\alpha,\,\beta$ are given by (\ref{62}). Therefore,
$$y(t)\leq C (T^\ast-t)^{-\alpha},\qquad z(t)\leq C (T^\ast-t)^{-\beta}\qquad\hbox{ on $[0,T^\ast)$.}$$
For $0<\rho\le R/2$, since $u,\,v$ are radially symmetric and radially nonincreasing, we deduce that
\begin{eqnarray*}
  & &\rho^n\, u(t,\,\rho)\leq
  C \int_{B(0,\,R/2)}u(t,|x|)dx\leq C\int_{B(0,\,R/2)}u(t,|x|)\varphi_1(x)dx\leq C (T^\ast-t)^{-\alpha},\\
  & &\rho^n\, v(t,\,\rho)\leq
  C \int_{B(0,\,R/2)}v(t,|x|)dx\leq C\int_{B(0,\,R/2)}
  v(t,|x|)\varphi_1(x)dx\leq C (T^\ast-t)^{-\beta}.
\end{eqnarray*}
The case when $R/2<\rho<R$ then follows from the radial nonincreasing property. 
This completes the proof.
 \end{proof}

\section{A non-degeneracy criterion for blow-up points}
\setcounter{equation}{0} 

The main objective of this subsection is the following result,
which gives a sufficient, local smallness condition, at any given time sufficiently close to $T^\ast$,
for excluding blow-up at a given point different from the origin.

\begin{pro}\label{15}
Let $\Omega= B(0,\,R),$ $p,\,q>1$, $\delta >0$.
Assume (\ref{initialdata})--(\ref{monotone}), (\ref{controle})--(\ref{controle1b})
 and let the solution $(u,\,v)$ of (\ref{1}) satisfy $T^\ast<\infty$.
Let $d_0,\,d_1$ satisfy $0<d_1<d_0<R$.
There exist $\eta,\, \tau_0>0$ such that if, for some
$t_1\in[T^\ast-\tau_0,\,T^\ast)$, we have
\begin{eqnarray}\label{17}
(T^\ast-t_1)^\alpha u(t_1,\,d_1)\leq \eta
\quad\hbox{ and}\quad
(T^\ast-t_1)^\beta v(t_1,\,d_1)\leq \eta,
\end{eqnarray}
then $d_0$ is not a blow-up point of
$(u,\,v)$, i.e. $(u,\,v)$ is uniformly bounded in the neighborhood
of $(T^\ast,\,d_0).$
Here, the numbers $\eta,\,\tau_0$ depend only on $p,\,q,\,r,\,s,\,\delta,\,c_1,\,c_2,\,d_0,\,d_1,\,n,\,R,\,T^\ast$.
\end{pro}

 As in \cite{souplet}, the proof uses similarity variables and delayed smoothing effects.
However, as explained in Section~1.2, a new difficulty arises, caused by the
absence of global type~I information on the blow-up rate. For this reason, we
consider only radial and radially decreasing solutions (whereas the
analogous criterion in \cite{souplet} was established for any
solution). In this more delicate situation, the current formulation,
slightly different from that in \cite{souplet}, turns out to be more
convenient. Namely, instead of expressing the local non-blow-up
criterion itself with the weighted $L^1$ norm of rescaled solution,
it is expressed
in terms of pointwise smallness on $((T^\ast-t)^\alpha
u,(T^\ast-t)^\beta v)$ at a point $d_1<d_0$ and at some time close to $T^*$.

\subsection{Similarity variables and delayed smoothing effects}
In view of the proof of Proposition~\ref{15}
we introduce the well-known similarity variables (cf.~\cite{kohn1}).
More precisely, for any given $d\in\mathbb{R}$, we define the
(one-dimensional) similarity variables around $(T^\ast,\,d)$,
associated with $(t,\,\rho)\in (0,\,T^\ast)\times\mathbb{R}$, by:
\begin{equation}\label{defsimilvar}
\sigma=-\log(T^\ast-t)\in [\hat\sigma,\,\infty),\qquad
\theta=\frac{\rho-d}{\sqrt{T^\ast-t}}=e^{\sigma/2}(\rho-d) \in \mathbb{R},
\end{equation}
where $\hat\sigma=-\log T^\ast$.
For given $\delta>0$, let $U$ be a (classical) solution of
$$U_t-\delta U_{\rho\rho}=H(t,\,\rho),\quad 0<t<T^*,\ \rho\in \mathbb{R}$$
(where the smooth functions $H$ will be specified later).
Then
$$V=V_d(\sigma,\,\theta)=(T^\ast-t)^{\alpha}U(t,\,y)
= e^{-\alpha\sigma}U\bigl(T^\ast-e^{-\sigma},\,d+\theta e^{-\sigma/2}\bigr)$$
is a solution of
\begin{equation}\label{eqsimilV}
V_\sigma-\mathcal{L_{\delta}}V+\alpha V=
e^{-(\alpha+1)\sigma}H\bigl(T^\ast-e^{-\sigma},\,d+\theta e^{-\sigma/2}\bigr),
 \quad \sigma>\hat\sigma,\ \theta\in \mathbb{R},
\end{equation}
where
\begin{eqnarray*}
 &&\mathcal{L_{\delta}}=\delta\partial^2_\theta -\frac{\theta }{2}\partial_\theta =\delta K_{\delta}^{-1}\partial_\theta (K_{\delta}\partial_\theta ),\qquad
 K_{\delta}(\theta )=(4\pi\delta)^{-1/2}e^{\frac{-\theta^2}{4\delta}}.
\end{eqnarray*} 
We denote by
$(T_{\delta}(\sigma))_{\sigma\geq0}$ the semigroup associated with
$\mathcal{L}_{\delta}$. More precisely, for each $\phi\in
L^\infty(\mathbb{R}),$ we set
$T_{\delta}(\sigma)\phi:=w(\sigma,\,.)$, where $w$ is the unique
solution of
\begin{equation}\label{43}
\left\{
  \begin{array}{ll}
    w_\sigma=\mathcal{L}_{\delta}w, & \hbox{ } \theta \in\mathbb{R},\,\,\sigma>0,\\
    w(0,\,\theta )=\phi(\theta ), & \hbox{ }\theta \in\mathbb{R}.
  \end{array}
\right.
\end{equation}
For any $\phi\in L^\infty(\mathbb{R}),$ we put
\begin{equation*}
    \|\phi\|_{L_{K_{\delta}}^m}=\left(\int_{\mathbb{R}}|\phi(\theta )|^mK_{\delta}(\theta )d\theta \right)^{1/m},\quad 1\leq m<\infty.
\end{equation*}
 Let $1\leq k<m<\infty$ and $\delta>0$, then, by Jensen's inequality,
 \begin{eqnarray}\label{9}
        & &\|\phi\|_{L_{K_{\delta}}^k}\leq 
        \|\phi\|_{L_{K_{\delta}}^m},\quad 1\leq
        k<m<\infty.
     \end{eqnarray}
The semigroups $(T_{\delta}(\sigma))_{\sigma\geq0}$ have the following properties,
which will be useful when dealing with system (\ref{1}) with unequal diffusivities:

\begin{lem}\label{66}
\begin{enumerate}
  \item (Contraction) For any $1\leq m<\infty$, we have
  \begin{equation}\label{11}
\|T_{\delta}(\sigma)\phi\|_{L_{K_{\delta}}^m}\leq\|\phi\|_{L_{K_{\delta}}^m},
 \quad\; \mbox{for all}\; \delta>0,\; \sigma\geq0,\,\phi\in L^\infty(\mathbb{R}).
\end{equation}
Moreover, for all $0<\delta\le \lambda<\infty$, we have 
  \begin{equation}\label{45}
 \|T_\delta(\sigma)\phi\|_{L_{K_\lambda}^m}\leq\Bigl(\frac{\lambda}{\delta}\Bigr)^{1/2}\|\phi\|_{L_{K_\lambda}^m},
 \quad\; \mbox{for all}\; \sigma\geq0,\,\phi\in L^\infty(\mathbb{R}).
 \end{equation}
  \item (Delayed regularizing effect) For any $1\leq k<m<\infty,$ there exist $\hat C, \sigma^\ast>0$ such that
 \begin{equation}\label{12}
\|T_{\delta}(\sigma)\phi\|_{L_{K_{\delta}}^m}\leq \hat C\|\phi\|_{L_{K_{\delta}}^k},
 \quad\; \mbox{for all}\;  \delta>0,\ \sigma\geq \sigma^\ast,\,\phi\in L^\infty(\mathbb{R}).
\end{equation}
Moreover, for all $0<\delta\le \lambda<\infty$, we have 
  \begin{equation}
 \|T_\delta(\sigma)\phi\|_{L_{K_\lambda}^m}\leq \hat C\Bigl(\frac{\lambda}{\delta}\Bigr)^{1/2}\|\phi\|_{L_{K_\lambda}^k},
 \quad\; \mbox{for all}\;\sigma\geq \sigma^\ast,\,\phi\in L^\infty(\mathbb{R}).
 \end{equation}
\end{enumerate}
\end{lem}

\begin{proof}
We put $\overline{w}(\sigma,\,\theta)=
(T_{\delta}(\sigma)\phi)(\sqrt{\delta}\,\theta).$ Then, by
(\ref{43}), it follows that $\overline{w}$ is the solution of
\begin{equation*}
\left\{
  \begin{array}{ll}
    \overline{w}_\sigma=\mathcal{L}_{1}\overline{w}, & \hbox{ } \theta \in\mathbb{R},\,\,\sigma>0,\\
    \overline{w}(0,\,\theta)=\phi(\sqrt{\delta}\,\theta), & \hbox{ }\theta \in\mathbb{R}.
  \end{array}
\right.
\end{equation*}
 Then \begin{eqnarray}\label{44}
 & &\overline{w}(\sigma,\,\theta)=\bigl [T_1(\sigma)\phi(\sqrt{\delta}\,.)\bigr](\theta).
 \end{eqnarray}
By (\ref{44}) and \cite[Lemma 3.1(i), p.176]{souplet}, we obtain
 \begin{eqnarray*}
 \|T_{\delta}(\sigma)\phi\|_{L_{K_{\delta}}^m}\hspace{-0,6cm}& &=\|(T_{\delta}(\sigma)\phi)(\sqrt{\delta}\,.)\|_{L_{K_{1}}^m}
=\|T_{1}(\sigma)\phi(\sqrt{\delta}\,.)\|_{L_{K_{1}}^m}\\
 \hspace{-0,6cm}& &\leq\|\phi(\sqrt{\delta}\,.)\|_{L_{K_{1}}^m}
 =\|\phi\|_{L_{K_{\delta}}^m},\quad \hbox{for all}\; \sigma\geq0.
 \end{eqnarray*}
Let next $0<\delta\le \lambda<\infty$. Denote by
$(S_{\delta}(t))_{t\geq0}$ the semigroup associated with $\delta \partial^2_y$ in $\mathbb{R}$
 and let the functions $u(t,\,y)$ and $w(\sigma,\,\theta)$ be related by the
 following backward self-similar transformation (with $ T^*=1$, $d=0$):
$$
\sigma=-\log(1-t)\in [0,\,\infty),\qquad
\theta=\frac{y}{\sqrt{1-t}}=e^{\sigma/2}y \in \mathbb{R},\qquad
w(\sigma,\,\theta)=u(t,\,y).
$$
We have, for all $\sigma\geq0$,
 \begin{eqnarray*}
 \bigl|\bigl[T_\delta(\sigma)\phi \bigr](\theta)\bigr|= \bigl| \bigl[S_{\delta}(t)u_0 \bigr](y)\bigr|\hspace{-0,6cm}& &=\left|(4\pi \delta t)^{-1/2}\int_{\mathbb{R}} e^{\frac{-|y-z|^2}{4 \delta t}}u_0(z)dz\right|\\
 \hspace{-0,6cm}& &\leq\Bigl(\frac{\lambda}{\delta}\Bigr)^{1/2}(4\pi\lambda t)^{-1/2}\int_{\mathbb{R}} e^{\frac{-|y-z|^2}{4\lambda t}}\left|u_0(z)\right|dz\\
 \hspace{-0,6cm}& &=\Bigl(\frac{\lambda}{\delta}\Bigr)^{1/2} \bigl[S_{\lambda}(t)\bigl|u_0 \bigr| \bigr](y)
 =\Bigl(\frac{\lambda}{\delta}\Bigr)^{1/2}\bigl[T_{\lambda}(\sigma)\bigl|\phi\bigr| \bigr](\theta).
  \end{eqnarray*}
Inequality (\ref{45}) then follows from (\ref{11}).
 
To prove assertion $(2)$, we recall that, by e.g.~\cite[Lemma 3.1(ii), p.176]{souplet}, 
for any $1\leq k<m<\infty$, there exist $\hat C, \sigma^\ast>0$ such that
$$
\|T_{1}(\sigma)\phi\|_{L_{K_{1}}^m}\leq \hat C\|\phi\|_{L_{K_{1}}^k},
 \quad\; \mbox{for all}\; \sigma\geq \sigma^\ast,\,\phi\in L^\infty(\mathbb{R}).
$$
We may then argue similarly as for assertion $(1)$.
\end{proof}

\subsection{Proof of Proposition~\ref{15}}
The proof is long and technical. We split it in several steps. 
Assume $p\ge q$ without loss of generality, hence $\alpha\ge\beta$.

\smallskip
{\bf Step 1.} {\it Definition of suitably modifed solutions.} As
mentioned before we lack a global type~I blow-up estimate. However,
we have a local type~I blow-up estimate, away from the origin.
Indeed, by (\ref{28}) in Proposition \ref{65}, we know that
\begin{eqnarray}\label{16}
(T^\ast-t)^\alpha u(t,\,y)\leq N_0,\quad(T^\ast-t)^\beta v(t,\,y)\leq N_0,\qquad 0\le t<T^\ast,
\ d_1\le y<R,
\end{eqnarray}
with $N_0=M_0\,d_1^{-n}$. We shall thus truncate the radial domain and
consider suitably controlled extensions of the solution to the real
line. We first define the following extensions $\widetilde u, \widetilde v\ge 0$ of $u, v$ by setting:
\begin{equation}\label{defutilde}
\widetilde u(t,\,y):=
\begin{cases}
u(t,\,y), & y\in [d_1,\,R],\\
\noalign{\vskip 1mm}
0, & y\in \mathbb{R}\setminus [d_1,\,R],
\end{cases}
\qquad\hbox{ for any $t\in [0,\,T^\ast)$,}
\end{equation}
and $\widetilde v(t,\,y)$ similarly.

Next, let $M\ge N_0$ to be chosen below.
For given $t_0\in [0,\,T^\ast)$, let
$(\overline{u},\,\overline{v})=(\overline{u}(t_0;\cdot,\cdot),$ $\,\overline{v}(t_0;\cdot,\cdot))$
be the solution of the following auxiliary problem:
\begin{equation}\label{defbaruv}
\left\{
  \begin{array}{ll}
    \overline{u}_{t}-\delta\overline{u}_{yy}=F(\widetilde u,\,\widetilde v),
    &\,t_0<t<T^\ast,\ y \geq d_1,\\
    \overline{v}_{t}-\overline{v}_{yy}=G(\widetilde u,\,\widetilde v),
    &\,t_0<t<T^\ast,\ y \geq d_1,\\
     \overline{u}(t,\,d_1)=M(T^\ast-t)^{-\alpha}, &t_0<t<T^\ast,\\
     \overline{v}(t,\,d_1)=M(T^\ast-t)^{-\beta}, &t_0<t<T^\ast,\\
   \overline{u}(t_0,\,y)=\widetilde u(t_0,\,y), &y  \geq d_1,\\
    \overline{v}(t_0,\,y)=\widetilde v(t_0,\,y), &y  \geq d_1.
\end{array}
\right.
\end{equation}
It is clear that $\overline u, \,\overline v \ge 0$ exist on $[t_0,\,T^\ast)\times [d_1,\,\infty)$.
Also, using (\ref{16}) and $M\ge N_0$, we deduce from the maximum principle that
\begin{equation}\label{comptildebar}
\widetilde u\le \overline{u},\ \ \widetilde v\le \overline{v}
\quad\hbox{ on $[t_0,\,T^\ast)\times [d_1,\,\infty)$.}
\end{equation}
Now choosing 
\begin{equation}\label{choiceMM0}
M=\max\Bigl(N_0,\ c_2\alpha^{-1}(N_0^p+N_0^r+{T^\ast}^{\alpha+1}),
\ c_2\beta^{-1}(N_0^q+N_0^s+{T^\ast}^{\beta+1})\Bigr),
\end{equation}
where $c_2$ is from (\ref{controle})--(\ref{controle1}), and using (\ref{controle})--(\ref{controle1b}), (\ref{16}), (\ref{defutilde}), (\ref{choiceMM0}), we have
$$F(\widetilde u,\widetilde v)
\le c_2(\widetilde{v}^p+\widetilde{u}^r+1)\le c_2\bigl((N_0^p+N_0^r)(T^\ast-t)^{-\alpha-1}+1\bigr)
\le \alpha M(T^\ast-t)^{-\alpha-1}$$
and
$$G(\widetilde u,\widetilde v)
\le c_2(\widetilde{u}^q+\widetilde{v}^s+1)\le c_2\bigl((N_0^q+N_0^s)(T^\ast-t)^{-\beta-1}+1\bigr)
\le \beta M(T^\ast-t)^{-\beta-1}.$$
We may thus use $M(T^\ast-t)^{-\alpha}$ (resp., $M(T^\ast-t)^{-\beta}$)
as a supersolution of the inhomogeneous, linear heat equation in (\ref{defbaruv}),
verified by $\overline u$ (resp. $\overline v$) on $[t_0,\,T^\ast)\times [d_1,\,\infty)$,
and infer from the maximum principle that
\begin{equation}\label{estimbar2M1}
0\leq \overline u\le M(T^\ast-t)^{-\alpha}, \quad
 0\leq \overline v\le M(T^\ast-t)^{-\beta} \quad\hbox{ on
$[t_0,\,T^\ast)\times [d_1,\,\infty).$}
\end{equation}

We next extend $(\overline{u},\,\overline{v})$ by odd reflection
for $y<d_1$, i.e., we set:
 \begin{eqnarray*}
 & &\overline u(t,\,d_1-y) = 2M(T^\ast-t)^{-\alpha}-\overline u(t,\,d_1+y),\qquad t_0\le t<T^\ast,\ y>0,\\
 & &\overline v(t,\,d_1-y) = 2M(T^\ast-t)^{-\beta}-\overline v(t,\,d_1+y),\qquad t_0\le t<T^\ast,\ y>0.
  \end{eqnarray*}
 From (\ref{estimbar2M1}), along with (\ref{comptildebar}) and (\ref{defutilde}), we have
\begin{equation}\label{estimbar2M}
0\le \overline u\le 2M(T^\ast-t)^{-\alpha}, \quad
0\le \overline v\le 2M(T^\ast-t)^{-\beta}
\quad\hbox{ on $[t_0,\,T^\ast)\times\mathbb{R}$}
\end{equation}
and 
\begin{equation}\label{comptildebar2}
\widetilde u\le \overline u, \quad
\widetilde v\le \overline v
\quad\hbox{ on $[t_0,\,T^\ast)\times\mathbb{R}.$}
\end{equation}
It is easy to see that the functions
$\overline u,\,  \overline v\in C^{1,2}((t_0,\,T^\ast)\times \mathbb{R})$
and that we have
\begin{equation*}
\left\{
  \begin{array}{ll}
    \overline{u}_{t}-\delta\overline{u}_{yy}=F_1(t,\,y),
    &\,t_0<t<T^\ast,\ y\in\mathbb{R},\\
    \overline{v}_{t}-\overline{v}_{yy}=G_1(t,\,y),
    &\,t_0<t<T^\ast,\ y\in\mathbb{R},\\
\end{array}
\right.
\end{equation*}
where
\begin{equation}\label{defF1}
F_1(t,\,y):=
\begin{cases}
2\alpha M(T^\ast-t)^{-\alpha-1}-F(\widetilde u,\,\widetilde v)(t,\,2d_1-y), & y<d_1,\\
\noalign{\vskip 1mm}
F(\widetilde u,\,\widetilde v)(t,\,y), & y\ge d_1,
\end{cases}
\end{equation}
\begin{equation}\label{defG1}
G_1(t,\,y):=
\begin{cases}
2\beta M(T^\ast-t)^{-\beta-1}-G(\widetilde u,\,\widetilde v)(t,\,2d_1-y), & y<d_1,\\
\noalign{\vskip 1mm}
G(\widetilde u,\,\widetilde v)(t,\,y), & y\ge d_1.
\end{cases}
\end{equation}
\smallskip
{\bf Step 2.} {\it Self-similar rescaling of modifed solutions.}
We now fix $d\in (d_1,\,d_0)$
(say, $d=(d_0+d_1)/2$) and pass to self-similar variables $(\sigma,\,\theta)$
around $(T^\ast,\,d)$, cf.~(\ref{defsimilvar}).
In these variables, we first define the rescaled solution $(\widetilde w,\,\widetilde z)=(\widetilde w_d, \,\widetilde z_d)$,
associated with the extended solution $(\widetilde u,\,\widetilde v)$,
namely,
\begin{equation}\label{deftildewz}
\left\{
  \begin{array}{llll}
    \widetilde w(\sigma,\,\theta)&=&(T^\ast-t)^{\alpha}\widetilde u(t,\,y),
    & \quad \hat\sigma\le \sigma<\infty,\ \theta\in \mathbb{R}, \\
    \widetilde z(\sigma,\,\theta)&=&(T^\ast-t)^{\beta}\widetilde v(t,\,y),
    & \quad \hat\sigma\le \sigma<\infty,\ \theta\in \mathbb{R}, \\
   \end{array}
\right.
\end{equation}
where $\hat\sigma=-\log T^\ast$.
For given $t_0\in [0,\,T^\ast)$ (cf.~Step~1), we also define
 $(\overline{w},\,\overline{z})=(\overline{w}_d(t_0;\cdot,\,\cdot),\,\overline{z}_d(t_0;\cdot,\,\cdot))$,
 associated with the modifed solution $(\overline{u}(t_0;\cdot,\,\cdot),\overline{v}(t_0;\cdot,\,\cdot))$, given by
\begin{equation}\label{defbarwz}
\left\{
  \begin{array}{llll}
\overline w(\sigma,\,\theta)&=&(T^\ast-t)^{\alpha}\overline u(t,\,y),
    & \quad \sigma_0\le \sigma<\infty,\ \theta\in \mathbb{R}, \\
\overline z(\sigma,\,\theta)&=&(T^\ast-t)^{\beta}\overline v(t,\,y),
    & \quad \sigma_0\le \sigma<\infty,\ \theta\in \mathbb{R}, \\
   \end{array}
\right.
\end{equation}
where $\sigma_0=-\log(T^\ast-t_0)$ $\ge \hat \sigma$.
At this point, we stress that $(\overline w,\,\overline z)$
depends on the choice of $\sigma_0$ (or $t_0$), whereas $(\widetilde w,\, \widetilde z)$ does not.
Actually, in Step 3, the $(\overline w,\,\overline z)$ will be used as auxiliary functions
in order to establish suitable estimates on $(\widetilde w, \,\widetilde z)$ itself.

Set $\ell=d-d_1>0$. Owing to (\ref{estimbar2M}), (\ref{comptildebar2}), we have
 \begin{equation}\label{bound-2M}
 \widetilde w\le \overline w\le 2M, \quad
\widetilde z\le \overline z\le 2M
\quad\hbox{ on $[\sigma_0,\,\infty)\times \mathbb{R}$}
 \end{equation}
 and, for all $\sigma\ge\hat\sigma$,
 \begin{equation}\label{monot-tilde}
 \theta\mapsto  \widetilde w(\sigma,\theta) \ \hbox{ and }\ \theta\mapsto  \widetilde z(\sigma,\theta)
 \ \hbox{ are nonincreasing for $\theta\in [-\ell e^{\sigma/2},\infty)$,}
  \end{equation}
due to (\ref{monot}).
Then, using (\ref{eqsimilV}), (\ref{defF1}), (\ref{defG1}),
 $\alpha+1=p\beta,$ $\beta+1=q\alpha$, $\alpha(r-1)-1=0$ and $\beta(s-1)-1=0$, we see that
$(\overline{w},\,\overline{z})$  is a solution of
\begin{equation}\label{13}
\hspace{-0,2cm}\left\{
  \begin{array}{llll}
 \overline{w}_\sigma-\mathcal{L}_{\delta}\overline{w}+\alpha \overline{w}&=&
 F_2(\sigma,\,\theta),
     & \quad \sigma_0  <  \sigma<\infty,\ \theta\in \mathbb{R}, \\
 \overline{z}_\sigma-\mathcal{L}_{1}\overline{z}+\beta \overline{z}&=& G_2(\sigma,\,\theta),
    & \quad \sigma_0 <\sigma<\infty,\ \theta\in \mathbb{R}, \\
  \end{array}
\right.
\end{equation}
where
\begin{multline}\label{defF2}
F_2(\sigma,\,\theta)=
e^{-(\alpha+1)\sigma}F_1\bigl(T^\ast-e^{-\sigma},\,d+\theta e^{-\sigma/2}\bigr) \\
\le  c_2\left(\widetilde z^p(\sigma)+\widetilde w^r(\sigma)+ e^{-(\alpha+1)\sigma} \right)+2\alpha
M\chi_{\{\theta<-\ell e^{\sigma/2}\}}
\end{multline}
and
\begin{multline}\label{defG2}
 G_2(\sigma,\,\theta)=
e^{-(\beta+1)\sigma}G_1\bigl(T^\ast-e^{-\sigma},\,d+\theta e^{-\sigma/2}\bigr)\\
\le  c_2\left(\widetilde w^q(\sigma)+ \widetilde z^s(\sigma)+e^{-(\beta+1)\sigma} \right)+2\beta
M\chi_{\{\theta<-\ell e^{\sigma/2}\}}
\end{multline}
Also, using the last two conditions in (\ref{defbaruv}),
 along with (\ref{deftildewz}), (\ref{defbarwz})
and (\ref{bound-2M}), we see that
\begin{equation}\label{split-barwz}
\overline w(\sigma_0)
\le\widetilde w(\sigma_0)+2M\chi_{\{\theta<-\ell e^{\sigma_0/2}\}}
\quad\hbox{and}\quad
\overline z(\sigma_0)
\le\widetilde z(\sigma_0)+2M\chi_{\{\theta<-\ell e^{\sigma_0/2}\}}.
\end{equation}

In the next steps, we shall estimate $(\widetilde w, \,\widetilde z)$
by using semigroup and delayed smoothing arguments.
As compared with the situation in \cite{souplet}, we have here additional terms which
come from the reflection procedure.
However, thanks to the self-similar change of variables,
whose center $d$ is shifted to the right of the reflection point $d_1$,
the contribution of these terms, as $\sigma\to\infty$, will be localized exponentially far away
at $-\infty$ in space and thus can be made arbitrarily small for $\tau_0$ small.
Also, the need to handle two semigroups, due to the different diffusivities, as well as added nonlinear terms,
cause some technical complications, which require for instance an additional interpolation argument.

\smallskip
 {\bf Step 3.} {\it First semigroup estimates for $(\widetilde w,\, \widetilde z)$}.
We claim that, for all $\sigma_0\ge \hat \sigma$ and $\sigma>0$, we have
\begin{multline}
\widetilde{w}(\sigma_0+\sigma)\leq
   e^{-\alpha \sigma}T_{\delta}(\sigma)
   \Bigl[\widetilde w(\sigma_0)+2M\chi_{\{\theta<-\ell e^{\sigma_0/2}\}}\Bigr]\\
+ c_2\displaystyle\int_0^\sigma
e^{-\alpha(\sigma-\tau)}T_{\delta}(\sigma-\tau)\left(\widetilde{z}^p(\sigma_0+\tau)+\widetilde w^r(\sigma_0+\tau)+e^{-(\alpha+1)(\sigma_0+\tau)}\right)d\tau\\
+2\alpha M\displaystyle \int_0^\sigma
e^{-\alpha(\sigma-\tau)}T_{\delta}(\sigma-\tau)
\chi_{\{\theta<-\ell e^{(\sigma_0+\tau)/2}\}}d\tau\quad\label{estim-wtilde}
  \end{multline}
  and
\begin{multline}
 \widetilde{z}(\sigma_0+\sigma)\leq
   e^{-\beta \sigma}T_{1}(\sigma)\Bigl[\widetilde z(\sigma_0)+2M\chi_{\{\theta<-\ell e^{\sigma_0/2}\}}\Bigr]\\
+c_2\displaystyle\int_0^\sigma
e^{-\beta(\sigma-\tau)}T_{1}(\sigma-\tau)\left(\widetilde{w}^q(\sigma_0+\tau)+\widetilde z^s(\sigma_0+\tau)+e^{-(\beta+1)(\sigma_0+\tau)}\right)d\tau\\
+2\beta M\displaystyle \int_0^\sigma e^{-\beta(\sigma-\tau)}T_{1}(\sigma-\tau)\chi_{\{\theta<-\ell e^{(\sigma_0+\tau)/2}\}}d\tau,\quad\label{estim-ztilde}
\end{multline}
  and that, moreover,
   \begin{multline}
\widetilde w(\sigma_0+\sigma)+\widetilde z(\sigma_0+\sigma)\leq e^{M_1\sigma}S(\sigma)\Bigl[\widetilde w(\sigma_0)+\widetilde z(\sigma_0)+4M\chi_{\{\theta<-\ell e^{\sigma_0/2}\}}\Bigr] \\
+c_2\displaystyle\int_0^\sigma e^{M_1(\sigma-\tau)}S(\sigma-\tau)\left[e^{-(\alpha+1)(\sigma_0+\tau)}+e^{-(\beta+1)(\sigma_0+\tau)}\right]d\tau\\
+2\alpha M \displaystyle\int_0^\sigma e^{M_1(\sigma-\tau)}S(\sigma-\tau) \chi_{\{\theta<-\ell e^{(\sigma_0+\tau)/2}\}}d\tau,\quad \label{estim-H}
 \end{multline}
 where
 $$(S(\sigma))_{\sigma\geq0}=(T_{\delta}(\sigma)+T_{1}(\sigma))_{\sigma\geq0}$$
 and
 $$M_1=c_2\max\bigl((2M)^{p-1},\, (2M)^{q-1},\, (2M)^{r-1},\, (2M)^{s-1}\bigr).$$
(Note that, as announced, estimates
(\ref{estim-wtilde})-(\ref{estim-H}) do not involve $(\overline w(t_0;\cdot,\cdot),\, \overline z(t_0;\cdot,\cdot))$ anymore.)

\smallskip
Let us first verify (\ref{estim-wtilde})-(\ref{estim-ztilde}). We fix $\sigma_0\ge \hat \sigma$ and consider
$(\overline{w},\,\overline{z})=\bigl(\overline{w}_d(t_0;\cdot,\,\cdot),\,$ $\overline{z}_d(t_0;\cdot,\,\cdot)\bigr)$,
defined in (\ref{defbarwz}) with $\sigma_0=-\log(T^\ast-t_0)$.
 we use (\ref{13}) and the variation of constants formula to write
$$
\overline{w}(\sigma_0+\sigma)=e^{-\alpha \sigma}T_{\delta}(\sigma)\overline w(\sigma_0)
 +\,\int_0^\sigma
e^{-\alpha(\sigma-\tau)}T_{\delta}(\sigma-\tau)F_2(\sigma_0+\tau,\,\cdot)d\tau
$$
for all $\sigma>0$, hence, by (\ref{defF2}),
\begin{multline}
\overline{w}(\sigma_0+\sigma)\leq
   e^{-\alpha \sigma}T_{\delta}(\sigma)\overline{w}(\sigma_0)
   +2\alpha M\displaystyle \int_0^\sigma e^{-\alpha(\sigma-\tau)}T_{\delta}(\sigma-\tau)
\chi_{\{\theta<-\ell e^{(\sigma_0+\tau)/2}\}}d\tau\\
+c_2\displaystyle \int_0^\sigma
e^{-\alpha(\sigma-\tau)}T_{\delta}(\sigma-\tau)\left(\widetilde{z}^p(\sigma_0+\tau)+\widetilde w^r(\sigma_0+\tau)+e^{-(\alpha+1)(\sigma_0+\tau)}\right)d\tau.
\quad\label{estim-wbar}
  \end{multline}
Similarly, by exchanging the roles of
$\overline{w}$, $\widetilde{w}$, $p,$ $r,$ $\alpha,$ and $\overline{z},$ $\widetilde{z},$ $q$, $s,$ $\beta$,
we obtain
\begin{multline}
\overline{z}(\sigma_0+\sigma)\leq
   e^{-\beta \sigma}T_{1}(\sigma)\overline{z}(\sigma_0)
   +2\beta M\displaystyle \int_0^\sigma e^{-\beta(\sigma-\tau)}T_{1}(\sigma-\tau)
\chi_{\{\theta<-\ell e^{(\sigma_0+\tau)/2}\}}d\tau\\
+c_2\displaystyle \int_0^\sigma
e^{-\beta(\sigma-\tau)}T_{1}(\sigma-\tau)\left(\widetilde{w}^q(\sigma_0+\tau)+\widetilde z^s(\sigma_0+\tau)+e^{-(\beta+1)(\sigma_0+\tau)}\right)d\tau.
\quad\label{estim-zbar}
  \end{multline}
Inequalities (\ref{estim-wtilde})-(\ref{estim-ztilde}) then follow from  (\ref{estim-wbar}),  (\ref{estim-zbar}), (\ref{bound-2M}) and (\ref{split-barwz}).

To verify (\ref{estim-H}), we set $H:=\overline{w}+\overline{z}$.
Adding up (\ref{estim-wbar}) and  (\ref{estim-zbar}), and recalling $\alpha\ge\beta$, we easily get
\begin{equation}\label{ineqhatH1}
H(\sigma_0+\sigma)\leq S(\sigma)H(\sigma_0)+\int_0^\sigma S(\sigma-\tau)
\bigl[M_1 H(\sigma_0+\tau)+D(\tau)\bigr]\, d\tau,\quad \sigma\ge 0,
\end{equation}
where
$$
D(\tau,\cdot)= c_2\bigl[ e^{-(\alpha+1)(\sigma_0+\tau)}+e^{-(\beta+1)(\sigma_0+\tau)}\bigr]
+2\alpha M\chi_{\{\theta<-\ell e^{(\sigma_0+\tau)/2}\}},\quad \tau\ge 0.
$$
Set
\begin{equation}\label{ineqhatH2}
\widehat H(\sigma_0+\sigma):=
e^{M_1\sigma}S(\sigma)H(\sigma_0)+ \int_0^\sigma e^{M_1(\sigma-\tau)}S(\sigma-\tau)D(\tau)d\tau,
\quad \sigma\ge 0.
\end{equation}
By direct computation, using the semigroup properties of $(S(\sigma))_{\sigma\ge 0}$ and Fubini's theorem, we see that
\begin{equation}\label{ineqhatH3}
 \widehat H(\sigma_0+\sigma)=S(\sigma)H(\sigma_0)+\int_0^\sigma S(\sigma-\tau)
\bigl[M_1  \widehat H(\sigma_0+\tau)+D(\tau)\bigr]\, d\tau,\quad \sigma>0.
\end{equation}
Combining (\ref{ineqhatH1}), (\ref{ineqhatH3}) and using the positivity-preserving property of $(S(\sigma))_{\sigma\geq0}$, we obtain
\begin{equation}\label{ineqhatH4}
[H-\widehat H]_+(\sigma_0+\sigma)\le M_1\int_0^\sigma S(\sigma-\tau)[H-\widehat H]_+(\sigma_0+\tau)\, d\tau,\quad \sigma>0.
\end{equation}
Letting now $\bar\delta=\max(\delta,1)$ and $K=K_{\bar\delta}$, 
we deduce from (\ref{45}) in Lemma \ref{66} that
 \begin{equation}
 \label{51} 
\|S(\sigma)\phi\|_{L_{K}^k}\leq
\widetilde C\|\phi\|_{L_{K}^k},\quad \sigma\geq 0,\,\phi\in
L^\infty(\mathbb{R}), \, 1\le k<\infty,
 \end{equation}
with $\widetilde C=\widetilde C(\delta)\ge 1$. Therefore, it follows from (\ref{ineqhatH4}) that
$$
\bigl\|[H-\widehat H]_+(\sigma_0+\sigma)\bigr\|_{L_{K}^1}\le
\widetilde C M_1\int_0^\sigma \bigl\|[H-\widehat H]_+(\sigma_0+\tau)\bigr\|_{L_{K}^1}\, d\tau,\quad \sigma>0,
$$
and we infer from Gronwall's Lemma that $H(\sigma_0+\sigma)\le\widehat H(\sigma_0+\sigma)$
for all $\sigma\ge 0$.
Inequality (\ref{estim-H}) then follows from \eqref{bound-2M} and (\ref{split-barwz}).

\smallskip
{\bf Step  4.} {\it Small time estimate of rescaled solutions.} At this point, we set, as before,
$\bar\delta=\max(\delta,1)$ and $K=K_{\bar\delta}$, and we fix
\begin{equation}\label{condmlarge}
m>\max\Bigl[p,\,q,\,s,\,r,\,1+r(r-1)(\alpha-\beta)\Bigr]
\end{equation}
and let $\sigma^\ast$ be given by Lemma~\ref{66}(2), with $k=1$.
We note that, by Lemma \ref{66}, we have
 \begin{equation}\label{57}
\|S(\sigma)\phi\|_{L_{K}^m}\leq
\widetilde C_0 \|\phi\|_{L_{K}^1},\quad \sigma\geq \sigma^\ast,\,\phi\in L^\infty(\mathbb{R}),
\end{equation}
with $\widetilde C_0=\widetilde C_0(p,\,q,\,s,\,r,\delta)\ge 1$.
Also, by (\ref{51}), we have
\begin{eqnarray}
\hspace{-1.4cm}\|S(\sigma)\chi_{\{\theta<-A\}}\|_{L_{K}^k}
\hspace{-0.5cm}& &\le\widetilde C \|\chi_{\{\theta<-A\}}\|_{L_{K}^k} =
\widetilde C\left((4\pi\overline\delta)^{-1/2}\int_{-\infty}^{-A} \exp\Bigl({\frac{-\theta^2}{4\overline\delta}}\Bigr)\,d\theta\right)^{1/k}
 \notag \\
\hspace{-0.5cm}& &\le C_0\exp(-(8k\overline\delta)^{-1}A^2),
\quad\hbox{ for all $A>0$ and $1\le k\le m$,} \label{estim-indicator}
\end{eqnarray}
with $C_0=C_0(p,\,q,\,s,\,r,\delta)\ge 1$.

Let $\eta>0$.
We claim that there exists $\tau_1\in(0,\,T^\ast)$, depending only on $\eta$ and and on the parameters
\begin{equation}\label{listparam}
p,\,q,\,r,\,s,\,\delta,\,c_1,\,c_2,\,d_0,\,d_1,\,n,\,R,\,T^\ast,
\end{equation}
such that:
    \begin{eqnarray}
\hbox{ For any $t_1\in[T^\ast-\tau_1,\,T^\ast)$ satisfying (\ref{17}) and $\sigma_1=-\log(T^*-t_1)$, we have}
\notag\\
\|\widetilde{w}(\sigma_1+\sigma)\|_{L_{K}^1}+\|\widetilde{z}(\sigma_1+\sigma)\|_{L_{K}^1}\leq \widetilde C_1 \eta,
\qquad 0<\sigma\leq\sigma^\ast,
   \label{21}
    \end{eqnarray}
with $\widetilde C_1=3\widetilde C e^{M_1\sigma^{\ast}}>0.$

To prove the claim, we choose $\sigma_0=\sigma_1$ in (\ref{estim-H}).
Observe that, by assumption (\ref{17}) and owing to (\ref{monot}),
we have $\widetilde w(\sigma_1,\, \cdot),\, \widetilde z(\sigma_1,\, \cdot)\le\eta$
 on $\mathbb{R}$, hence
\begin{equation}\label{cond2eta}
\|\widetilde{w}(\sigma_1)\|_{L_{K}^1}
+\|\widetilde {z}(\sigma_1)\|_{L_{K}^1}\le 2\eta.
\end{equation}
 Using (\ref{estim-H}), (\ref{51}), (\ref{estim-indicator}), (\ref{cond2eta}),
$e^{\sigma_1}=(T^\ast-t_1)^{-1}\ge \tau_1^{-1}$ and assuming $\tau_1<1$,
we deduce that, for $0\le \sigma\le\sigma^\ast$,
\begin{eqnarray*}
&&\hspace{-1.5cm}\|\widetilde{w}(\sigma_1+\sigma)\|_{L_{K}^1}+\|\widetilde {z}(\sigma_1+\sigma)\|_{L_{K}^1}\\
\noalign{\vskip 1mm}
 &\le& 2\widetilde C e^{M_1\sigma^{\ast}}\eta +4\widetilde C C_0Me^{M_1\sigma^{\ast}} \exp(- (8\bar\delta\tau_1)^{-1}\ell ^2)\\
 \noalign{\vskip 1mm}
&&+2 C \widetilde C\sigma^\ast e^{M_1\sigma^{\ast}}\tau_1^{\beta+1}+2\alpha C_0 \widetilde C\sigma^{\ast} Me^{M_1\sigma^{\ast}
\exp\bigl(- (8\bar\delta\tau_1)^{-1}\ell^2\bigr)}\\
\noalign{\vskip 1mm}
&\le&  2\widetilde C e^{M_1\sigma^{\ast}}\bigr[\eta+c_2\sigma^\ast\tau_1^{\beta+1}
+C_0 M (\alpha\sigma^\ast+2)\exp(-(8\bar\delta\tau_1)^{-1}\ell^2)\bigr].
\end{eqnarray*}
For $\tau_1\in(0,\,T^\ast)$ sufficiently small, depending only on $\eta$ and on the parameters
in~(\ref{listparam}), we finally get (\ref{21}) with
$\widetilde C_1=3\widetilde Ce^{M_1\sigma^{\ast}}$.

    \smallskip
{\bf Step 5.} {\it Large time estimate of rescaled solutions.}
We claim that there exist $\eta>0$ and $\tau_0\in(0,\,\tau_1(\eta)]$,
depending only on the parameters in~(\ref{listparam}), such that:
     \begin{eqnarray}
\hbox{ for any $t_1\in[T^\ast-\tau_0,\,T^\ast)$ satisfying (\ref{17}), we have
$\mathcal{A}_{\eta,\,t_1}=(0,\,\infty)$,}
    \end{eqnarray}
where $\sigma_1=-\log(T^*-t_1)$ and
$$\mathcal{A}_{\eta,t_1}=\Bigl\{\sigma>0\,;\ 
e^{\alpha\tau}\|\widetilde{w}(\sigma_1+\sigma^\ast+\tau)\|_{L_{K}^1}+e^{\beta\tau}\|\widetilde{z}(\sigma_1+\sigma^\ast+\tau)\|_{L_{K}^1}\leq
2\widetilde C\widetilde C_1\eta,\ \ \tau\in[0,\,\sigma]\Bigr\}.$$

First observe that $\mathcal{A}_{\eta,\,t_1}\neq\emptyset$, due to
(\ref{21}) and the continuity of the function $\sigma\mapsto
e^{\alpha
\sigma}\|\widetilde{w}(\sigma_1+\sigma^\ast+\sigma)\|_{L_{K}^1}+e^{\beta
\sigma}\|\widetilde{z}(\sigma_1+\sigma^\ast+\sigma)\|_{L_{K}^1}$.
We denote  
$$\overline T=\sup \mathcal{A}_{\eta,\,t_1} \in (0,\,\infty].$$
Assume for contradiction that $\overline T<\infty$.
Then by (\ref{21}), recalling that $\alpha\ge\beta$, we have
\begin{eqnarray}\label{22}
  \hspace{-2cm}& &\|\widetilde{w}(\sigma_1+\sigma^\ast+\sigma)\|_{L_{K}^1}+\| \widetilde{z}(\sigma_1+\sigma^\ast+\sigma)\|_{L_{K}^1}\leq 2 \widetilde{C} \widetilde C_1\eta e^{-\beta
  \sigma},\quad -\sigma^\ast\leq \sigma\leq \overline T.
\end{eqnarray}
For $0\leq\tau\leq \overline T$,
we apply (\ref{estim-H}) with $\sigma_0=\sigma_1+\tau$ and $\sigma=\sigma^\ast$.
Using (\ref{51}), (\ref{57}),  (\ref{estim-indicator}), (\ref{cond2eta}), (\ref{22}),
$e^{\sigma_1}=(T^\ast-t_1)^{-1}\ge \tau_0^{-1}$ and assuming $\tau_0<1$, we get
\begin{eqnarray*}
&&\|\widetilde{w}(\sigma_1+\sigma^\ast+\tau)\|_{L_{K}^m}
+\|\widetilde{z}(\sigma_1+\sigma^\ast+\tau)\|_{L_{K}^m}\\
\noalign{\vskip 1mm}
&\leq& 2\widetilde C_0 e^{M_1\sigma^\ast}
\left(\|\widetilde{w}(\sigma_1+\tau)\|_{L_{K}^1}
+\|\widetilde{z}(\sigma_1+\tau)\|_{L_{K}^1}\right)
+ 8C_0Me^{M_1\sigma^\ast}\exp\bigl(-( 8\bar\delta \tau_0 m)^{-1}\ell^2e^{\tau}\bigr) \\
\noalign{\vskip 1mm}
&&+4 c_2  \widetilde C \sigma^\ast e^{M_1\sigma^{\ast}}\tau_0^{\beta+1}e^{-(\beta+1)\tau}
+4\alpha M \widetilde C C_0  \sigma^\ast
e^{M_1\sigma^\ast}\exp\bigl(-( 8\bar\delta \tau_0 m)^{-1}\ell^2e^{\tau}\bigr)\\
\noalign{\vskip 1mm}
&\leq&
 4\widetilde C_1\widetilde C_0\widetilde C e^{M_1\sigma^\ast} \eta e^{- \beta(\tau-\sigma^\ast)}+4 c_2  \widetilde C \sigma^\ast e^{M_1\sigma^{\ast}}\tau_0^{\beta+1}e^{-(\beta+1)\tau}\\
 \noalign{\vskip 1mm}
&&+4C_0  (2+\alpha\widetilde C\sigma^\ast)Me^{M_1\sigma^\ast}\exp\bigl(-( 8\bar\delta \tau_0 m)^{-1}\ell^2e^{\tau}\bigr).
     \end{eqnarray*}
Put $\widetilde C_2=5\widetilde C_1\widetilde C_0\widetilde C e^{(M_1+\beta)\sigma^\ast}$.
For $\tau_0\in (0,\tau_1(\eta)]$ sufficiently small, depending only on $\eta$ and on the parameters in~(\ref{listparam}), 
it follows that
\begin{eqnarray}\label{23}
       & &\|\widetilde{w}(\sigma_1+\sigma^\ast+\tau)\|_{L_{K}^m}+\|\widetilde{z}(\sigma_1+\sigma^\ast+\tau)\|_{L_{K}^m} \leq
       \widetilde C_2 \eta e^{- \beta\tau},\ \ 0\le\tau\leq \overline T.
     \end{eqnarray}
     
Next let $0<\sigma\leq \overline T$. Now using (\ref{estim-wtilde}) with $\sigma_0=\sigma_1+\sigma^\ast$, (\ref{51}),  (\ref{estim-indicator}), $T_\delta(\sigma)\le S(\sigma)$ and $e^{\sigma_1}\ge \tau_0^{-1}$, we obtain
\begin{eqnarray*}
e^{\alpha\sigma}\|\widetilde{w}(\sigma_1+\sigma^\ast+\sigma)\|_{L_{K}^1}
&\leq&\|T_{\delta}(\sigma)\widetilde{w}(\sigma_1+\sigma^\ast)\|_{L_{K}^1}
 +2M\|T_{\delta}(\sigma)\chi_{\{\theta<-\ell e^{(\sigma_1+\sigma^\ast)/2}\}}\|_{L_{K}^1} \\
&&+c_2\int_0^\sigma e^{\alpha\tau}\|T_{\delta}(\sigma-\tau)\widetilde{z}^p(\sigma_1+\sigma^\ast+\tau)\|_{L_{K}^1} d\tau\\
&&+c_2\int_0^\sigma e^{\alpha\tau}\|T_{\delta}(\sigma-\tau)\widetilde{w}^r(\sigma_1+\sigma^\ast+\tau)\|_{L_{K}^1} d\tau\\
&&+c_2\int_0^\sigma e^{\alpha\tau}\|T_{\delta}(\sigma-\tau)e^{-(\alpha+1)(\sigma_1+\sigma^\ast+\tau)}\|_{L_{K}^1} d\tau\\
&&+2\alpha M\displaystyle \int_0^\sigma e^{\alpha\tau}\|T_{\delta}(\sigma-\tau)
\chi_{\{\theta<-\ell e^{(\sigma_1+\sigma^\ast+\tau)/2}\}}\|_{L_{K}^1}d\tau
     \end{eqnarray*}
     hence,
 \begin{eqnarray*}
e^{\alpha\sigma}\|\widetilde{w}(\sigma_1+\sigma^\ast+\sigma)\|_{L_{K}^1}
& \leq& \widetilde C \|\widetilde{w}(\sigma_1+\sigma^\ast)\|_{L_{K}^1}
+2C_0 M\exp(-(8\overline\delta\tau_0)^{-1}\ell ^2)\\
&&+c_2 \widetilde C \int_0^\sigma e^{\alpha\tau}\|\widetilde{z}^p(\sigma_1+\sigma^\ast+\tau)\|_{L_{K}^1} d\tau\\
&&+c_2 \widetilde C \int_0^\sigma  e^{\alpha\tau}\|\widetilde{w}^r(\sigma_1+\sigma^\ast+\tau)\|_{L_{K}^1} d\tau\\  
&&+c_2 \widetilde C  \tau_0^{\alpha+1}+2\alpha C_0 M \displaystyle\int_0^\sigma e^{\alpha\tau}
\exp\bigl(-(8\overline\delta\tau_0)^{-1}\ell^2e^{\tau}\bigr)\,d\tau.
     \end{eqnarray*}
By taking $\tau_0$ possibly smaller (dependence as above), we may ensure that
$$c_2  \widetilde C \tau_0^{\alpha+1}+2C_0 M\exp(- (8\overline\delta\tau_0)^{-1}\ell ^2)
+2\alpha C_0 M \displaystyle\int_0^\infty e^{\alpha\tau}
\exp\bigl(-(8\overline\delta\tau_0)^{-1}\ell^2e^{\tau}\bigr)\,
d\tau\le \eta^2,$$
hence,
\begin{equation}\label{estim-ealpha}
      \begin{array}{lll}
      &e^{\alpha\sigma}\|\widetilde{w}(\sigma_1+\sigma^\ast+\sigma)\|_{L_{K}^1}
\leq\widetilde C\|\widetilde{w}(\sigma_1+\sigma^\ast)\|_{L_{K}^1}
+\eta^2 \\
                   \noalign{\vskip 2mm}
&\quad+c_2 \widetilde C\displaystyle \int_0^\sigma e^{\alpha\tau}\|\widetilde{z}(\sigma_1+\sigma^\ast+\tau)\|_{L_{K}^p}^p d\tau+c_2 \widetilde C \displaystyle \int_0^\sigma
 e^{\alpha\tau}\|\widetilde{w}(\sigma_1+\sigma^\ast+\tau)\|_{L_{K}^r}^r d\tau.
         \end{array}
  \end{equation}
 To estimate the last integral, setting $\nu=(m-r)/(m-1)\in (0,\,1)$ and interpolating between
(\ref{23}) and the fact that $\tau\in\mathcal{A}_{\eta,\,t_1}$,
we write
$$
      \begin{array}{ll}
\|\widetilde{w}(\sigma_1+\sigma^\ast+\tau)\|_{L_{K}^r}
&\le \|\widetilde{w}(\sigma_1+\sigma^\ast+\tau)\|_{L_{K}^1}^\nu
\|\widetilde{w}(\sigma_1+\sigma^\ast+\tau)\|_{L_{K}^m}^{1-\nu} \\
                   \noalign{\vskip 2mm}
&\le(2\widetilde C\widetilde C_1\eta e^{-\alpha\tau})^\nu(\widetilde C_2 \eta e^{-\beta\tau})^{1-\nu}=\widetilde C_3 \eta e^{-(\alpha\nu+\beta(1-\nu))\tau},
         \end{array}
         $$
with $\widetilde C_3=(2\widetilde C\widetilde C_1)^\nu \widetilde C_2^{1-\nu}$.
Using this, along with (\ref{9}) and (\ref{23}), we obtain
$$
      \begin{array}{lll}
      &e^{\alpha\sigma}\|\widetilde{w}(\sigma_1+\sigma^\ast+\sigma)\|_{L_{K}^1}\leq
 \widetilde C\|\widetilde{w}(\sigma_1+\sigma^\ast)\|_{L_{K}^1}+\eta^2\\
                   \noalign{\vskip 2mm}
&\qquad\qquad\qquad +c_2 \widetilde C (\widetilde C_2\eta)^p \displaystyle\int_0^\sigma e^{\alpha\tau}e^{-\beta p\tau} d\tau
+c_2 \widetilde C (\widetilde C_3\eta)^r \displaystyle\int_0^\sigma e^{\alpha\tau}e^{-(\alpha\nu+\beta(1-\nu))r\tau} d\tau.
         \end{array}
$$
Since $\alpha-\beta p=-1$, $\alpha=\alpha r-1$ and $\nu_1:=1-(\alpha-\beta)(1-\nu)r>0$,
owing to (\ref{condmlarge}), we deduce that
\begin{equation}\label{24}
e^{\alpha\sigma}\|\widetilde{w}(\sigma_1+\sigma^\ast+\sigma)\|_{L_{K}^1}\leq
\widetilde C\|\widetilde{w}(\sigma_1+\sigma^\ast)\|_{L_{K}^1}+\eta^2
+c_2 \widetilde C \widetilde C_2^p\eta^p+ \frac{c_2 \widetilde C \widetilde C_3^r\eta^r}{\nu_1}.
\end{equation}

Similarly as (\ref{estim-ealpha}), by using (\ref{estim-ztilde}) instead of (\ref{estim-wtilde}), we get
$$
      \begin{array}{lll}
      &e^{\beta\sigma}\|\widetilde{z}(\sigma_1+\sigma^\ast+\sigma)\|_{L_{K}^1}
\leq\widetilde C\|\widetilde{z}(\sigma_1+\sigma^\ast)\|_{L_{K}^1}
+\eta^2 \\
                   \noalign{\vskip 2mm}
&\quad+c_2 \widetilde C\displaystyle \int_0^\sigma e^{\beta\tau}\|\widetilde{w}(\sigma_1+\sigma^\ast+\tau)\|_{L_{K}^q}^q d\tau+c_2 \widetilde C \displaystyle \int_0^\sigma
 e^{\beta\tau}\|\widetilde{z}(\sigma_1+\sigma^\ast+\tau)\|_{L_{K}^s}^s d\tau.
         \end{array}
$$
Therefore, by (\ref{23}),
$$
      \begin{array}{lll}
      &e^{\beta\sigma}\|\widetilde{z}(\sigma_1+\sigma^\ast+\sigma)\|_{L_{K}^1}\leq
 \widetilde C\|\widetilde{z}(\sigma_1+\sigma^\ast)\|_{L_{K}^1}+\eta^2\\
                   \noalign{\vskip 2mm}
&\qquad\qquad\qquad +c_2 \widetilde C (\widetilde C_2\eta)^q \displaystyle\int_0^\sigma e^{\beta\tau}e^{-\beta q\tau} d\tau
+c_2 \widetilde C (\widetilde C_2\eta)^s \displaystyle\int_0^\sigma e^{\beta\tau}e^{-\beta s\tau} d\tau.
         \end{array}
$$
This time, the above interpolation is not necessary.
Indeed, using $\beta=\beta s-1$, we directly get
\begin{equation}\label{25}
e^{\beta\sigma}\|\widetilde{z}(\sigma_1+\sigma^\ast+\sigma)\|_{L_{K}^1}\leq
\widetilde C\|\widetilde{z}(\sigma_1+\sigma^\ast)\|_{L_{K}^1}+\eta^2
+\frac{c_2 \widetilde C \widetilde C_2^q\eta^q}{\beta(q-1)}+c_2 \widetilde C \widetilde C_2^{s}\eta^{s}.
\end{equation}
 
 Finally, for $\sigma=\overline T$ in (\ref{24}) and (\ref{25}), by definition of $\overline T$ and by using (\ref{21}) with $\sigma=\sigma^\ast$, we obtain
\begin{eqnarray*}
   \hspace{1,5cm}2\widetilde C\widetilde C_1\eta
      \hspace{-0,6cm} & &=e^{\alpha \overline T}\|\widetilde{w}(\sigma_1+\sigma^\ast+\overline T)\|_{L_{K}^1}
   +e^{\beta \overline T}\|\widetilde{z}(\sigma_1+\sigma^\ast+\overline T)\|_{L_{K}^1}\\
   \hspace{1,5cm}\hspace{-0,6cm} & &
   \leq \widetilde C \|\widetilde{w}(\sigma_1+\sigma^\ast)\|_{L_{K}^1}
   +\widetilde C\|\widetilde{z}(\sigma_1+\sigma^\ast)\|_{L_{K}^1}\\
   \hspace{1,5cm}\hspace{-0,6cm} & &\hspace{0,3cm}+2\eta^2
+c_2 \widetilde C \widetilde C_2^p\eta^p+ \frac{c_2 \widetilde C \widetilde C_3^r\eta^r}{ \nu_1}
+\frac{c_2 \widetilde C \widetilde C_2^q\eta^q}{\beta(q-1)}+c_2 \widetilde C \widetilde C_2^s\eta^s\\
  \hspace{1,5cm}\hspace{-0,6cm} & &
\leq \widetilde C \widetilde C_1\eta+C_4[\eta^2+\eta^p+\eta^r+\eta^q+\eta^s],
\end{eqnarray*}
hence
$\widetilde C\widetilde C_1\le C_4(\eta+\eta^{p-1}+\eta^{r-1}+\eta^{q-1}+\eta^{s-1})$,
where $C_4>0$ depends only on the parameters in~(\ref{listparam}).
Since $p,q,r,s>1$, choosing $\eta>0$ sufficiently small (which now fixes $\tau_0$), we reach a contradiction.
Consequently, $\overline T=\infty$
 and the claim is proved.

    \smallskip
{\bf Step 6.} {\it Conclusion.} Let $\eta,\,\tau_0$ be as in Step~5
and let $t_1\in[T^\ast-\tau_0,\,T^\ast)$ satisfy (\ref{17}).
It follows from  the definition of $\mathcal{A}_{\eta,\,t_1}$ that
\begin{equation}\label{Lambda0}
\Lambda_0=\sup_{\sigma\ge \sigma_1+\sigma^\ast}
\Bigl(e^{\alpha\sigma}\|\widetilde{w}(\sigma)\|_{L_{K}^1}
+e^{\beta\sigma}\|\widetilde{z}(\sigma)\|_{L_{K}^1}\Bigr)<\infty.
\end{equation}
Set $L:=\int_{-1}^0 K(\theta)\,d\theta>0$.  For all $t\in[\hat T^\ast-\ell^{-2},\,T^\ast)$, 
recalling (\ref{defsimilvar}), we have $\ell e^{\sigma/2}\ge 1$, hence
\begin{equation}\label{CompTildeMonot}
\widetilde{w}(\sigma,\,0)\le L^{-1}\int_{-1}^0\widetilde{w}(\sigma,\,\theta)K(\theta)\,d\theta,
\qquad
\widetilde{z}(\sigma,\,0)\le L^{-1}\int_{-1}^0\widetilde{z}(\sigma,\,\theta)K(\theta)\,d\theta,
\end{equation}
owing to (\ref{monot-tilde}).
Let then $\hat t_1=T^\ast-\min\bigl(\ell^{-2},e^{-(\sigma_1+\sigma^\ast)}\bigr)$. 
It follows from (\ref{defutilde}), (\ref{deftildewz}), (\ref{Lambda0}), (\ref{CompTildeMonot}) that, 
for all $t\in[\hat t_1,\,T^\ast)$, 
\begin{eqnarray*}
u(t,\,d)+v(t,\,d)
&=&e^{\alpha \sigma}\widetilde{w}(\sigma,\,0)+e^{\beta \sigma}\widetilde{z}(\sigma,\,0) \\
&\leq& 2L^{-1}\Bigl(e^{\alpha\sigma}\|\widetilde{w}(\sigma)\|_{L_{K}^1}
+e^{\beta\sigma}\|\widetilde{z}(\sigma)\|_{L_{K}^1}\Bigr) \le 2L^{-1}\Lambda_0.
\end{eqnarray*}
Using (\ref{monot}), we conclude that $d_0>d$ is not a blow-up point.
\qed


\section{Convergence of rescaled solutions to solutions of a system
of ordinary differential inequalities}
\setcounter{equation}{0} 

For given $\rho_1\in (0,R)$, we again switch to similarity variables around $(T^\ast,\,\rho_1)$,
already used in the previous section. Namely, we set:
 \begin{equation}\label{defsimilvar1}
\sigma=-\log(T^\ast-t),\qquad
\theta=\frac{\rho-\rho_1}{\sqrt{T^\ast-t}}=e^{\sigma/2}(\rho-\rho_1),
\end{equation}
and consider the rescaled solution $(W,Z)=(W_{\rho_1},Z_{\rho_1})$ 
associated with $(u,v)$:
 \begin{equation}\label{defsimilvar2}
 W(\sigma,\,\theta)=(T^\ast-t)^{\alpha}u(t,\,\rho),
\qquad
Z(\sigma,\,\theta)=(T^\ast-t)^{\beta}v(t,\,\rho),
\end{equation}
defined for $\sigma\in [\hat\sigma,\,\infty)$ with $\hat\sigma=-\log T^\ast$ and 
$\theta\in (-\rho_1 e^{\sigma/2},(R-\rho_1)e^{\sigma/2})$.

The goal of this section is to show that any such rescaled solution $(W,Z)$ behaves, in a suitable sense
as $\sigma\to\infty$ and $\theta\to\infty$,
like a (distribution) solution of the following system of ordinary differential inequalities:
\begin{equation}\label{27b}
\left\{
  \begin{array}{lll}
     &\phi'+\alpha \phi\ge c_1\psi^p,\\
          \noalign{\vskip 1mm}
    &\psi'+\beta \psi \ge c_1\phi^q
      \end{array}
\right.
\end{equation}
on the whole real line $(-\infty,\infty)$
(however, we shall eventually only use the fact that $(\phi,\psi)$ solves (\ref{27b}) on some bounded open interval).
Moreover, we single out a simple but crucial property of local interpendence of components 
for solutions of (\ref{27b}).

\begin{pro}\label{30lem}
Let $\Omega= B(0,\,R),$ $p,\,q>1$, $\delta >0$.
Assume (\ref{classical})--(\ref{monotone}), (\ref{controle})--(\ref{controle1b})
and let the solution $(u,\,v)$ of (\ref{1}) satisfy $T^\ast<\infty$.
Let $\rho_1\in (0,R)$ and let $(W,Z)$ be defined by (\ref{defsimilvar1})-(\ref{defsimilvar2}). 

(i) Then, for all sequence $\sigma_j\to\infty$, there exists a subsequence (not relabeled) such that, for each $\sigma\in\mathbb{R}$, 
 \begin{equation}\label{defphipsi}
 \phi(\sigma)=\lim_{\theta\to\infty}\Bigl(\lim_{j\to\infty}W(\sigma+\sigma_j,\theta)\Bigr),\qquad 
\psi(\sigma)=\lim_{\theta\to\infty}\Bigl(\lim_{j\to\infty}Z(\sigma+\sigma_j,\theta)\Bigr)
\end{equation}
exist and are finite, where the limits in $j$ are uniform for $(\sigma,\theta)$ 
in bounded subsets of $\mathbb{R}\times\mathbb{R}$, and the limits in $\theta$ are monotone nonincreasing.

(ii) The functions $\phi,\psi$ defined in (\ref{defphipsi}) belong to $BC(\mathbb{R})$ and $(\phi,\psi)$ is a nonnegative solution in $\mathcal{D}'(\mathbb{R})$ of system (\ref{27b}).

(iii) Let $I\subset \mathbb{R}$ be an open interval containing $0$.
For any nonnegative functions $\phi, \psi\in C(I)$ satisfying (\ref{27b}) in $\mathcal{D}'(I)$,
we have $\phi(0)=0$ if and only if $\psi(0)=0$.
\end{pro}

\begin{proof}
(i) Let $A=\min(\rho_1/2,R-\rho_1)>0$.
By (\ref{28}), we have that
 \begin{equation}\label{bddWZ}
 \hbox{$(W,\,Z)$ is bounded on the set
$D=\{(\sigma,\,\theta)\in \mathbb{R}\times\mathbb{R},\ \sigma> \hat\sigma,\ |\theta|\le Ae^{\sigma/2}\}$}
\end{equation}
and $(W,\,Z)$ solves the system
\begin{equation}\label{13b}
\hspace{-0,2cm}\left\{
  \begin{array}{llll}
 W_\sigma-\delta W_{\theta\theta} + \Bigl[\frac{\theta}{2}-\delta\frac{(n-1)e^{-\sigma/2}}{\rho_1+\theta e^{-\sigma/2}}\Bigr]
 \,W_{\theta}+\alpha W
 &=&e^{-(\alpha+1)\sigma} F(e^{\alpha\sigma}W,\,e^{\beta\sigma}Z)\\
 Z_\sigma-Z_{\theta\theta} +\Bigl[\frac{\theta}{2}-
 \frac{(n-1)e^{-\sigma/2}}{\rho_1+\theta e^{-\sigma/2}}\Bigr]\,Z_{\theta}+\beta Z
 &=& e^{-(\beta+1)\sigma} G(e^{\alpha\sigma}W,\,e^{\beta\sigma}Z)
  \end{array}
\right.
\ \hbox{in $D$.}
\end{equation}
Moreover, by (\ref{62}), (\ref{controle})--(\ref{controle1b}), it follows that
\begin{eqnarray}\label{cc}
c_1\,Z^p\leq 
e^{-(\alpha+1)\sigma} F\bigl(e^{\alpha\sigma}W,\,e^{\beta\sigma}Z\bigr)\leq 
c_2\bigl(Z^p+W^r+e^{-(\alpha+1)\sigma}\bigr).\\ \label{aa}
c_1 \,W^q \leq 
e^{-(\beta+1)\sigma} G\bigl(e^{\alpha\sigma}W,\,e^{\beta\sigma}Z\bigr)\leq
c_2\bigl(W^q+Z^s+e^{-(\beta+1)\sigma}\bigr).
\end{eqnarray}
Denoting the time-translates $W_j(\sigma,\,\theta):= W(\sigma+\sigma_j,\,\theta)$
and $Z_j(\sigma,\,\theta):=Z(\sigma+\sigma_j,\,\theta)$
and setting
$$\mu_j(\sigma,\theta)=\frac{(n-1)e^{-(\sigma+\sigma_j)/2}}{\rho_1+\theta e^{-(\sigma+\sigma_j)/2}},
\quad \eps_j(\sigma)=e^{-(\alpha+1)(\sigma+\sigma_j)},\quad
\widetilde\eps_j(\sigma)=e^{-(\beta+1)(\sigma+\sigma_j)},$$
we have, by (\ref{13b})-(\ref{aa}),
\begin{equation}\label{13c}
\hspace{-0,2cm}\left\{
  \begin{array}{lllll}
c_1Z_j^p
&\le& \partial_\sigma W_j-\delta \partial^2_\theta W_j
+\bigl[\frac{\theta}{2}-\delta\mu_j\bigr]\partial_\theta W_j+\alpha W_j
&\le& c_2\bigl(Z_j^p+W_j^r+\eps_j\bigr)\\
 \noalign{\vskip 1mm}
c_1W_j^q
&\le& \partial_\sigma Z_j- \partial^2_\theta Z_j
+\bigl[\frac{\theta}{2}-\mu_j\bigr]\partial_\theta Z_j+\beta Z_j
&\le& c_2\bigl(W_j^q+Z_j^s+\widetilde\eps_j\bigr)\\
  \end{array}
\right.
\ \hbox{ in $D$.}
\end{equation}

For each compact $Q$ of $\mathbb{R}\times\mathbb{R}$, the sequences 
$(Z_j^p+W_j^r+\eps_j)$ and $(W_j^q+Z_j^s+\widetilde\eps_j)$ 
are defined on $Q$ for $j$ large enough and, owing to (\ref{bddWZ}),
they are bounded in $L^m(Q)$ for each $m\in(1,\,\infty)$. 
Therefore, by (\ref{13c}) and parabolic estimates (see, e.g. \cite[p.438]{pavol}),
the sequences $(W_j)$ and $(Z_j)$ are bounded in $W^{1,2;\,m}(Q)$ for each compact $Q$ of
$\mathbb{R}\times\mathbb{R}$ and each $m\in(1,\,\infty)$. 
Fixing $\alpha\in (0,1)$ and using the compact embeddings $W^{1,2;\,m}(Q)\subset\subset C^{\alpha,1+(\alpha/2)}(\overline Q)$
for $m$ large, we deduce that, for some subsequence (not relabeled),
$(W_j, Z_j)$ converges,
 in $C^{\alpha,1+(\alpha/2)}$ for each compact $Q$ of
$\mathbb{R}\times\mathbb{R}$, to some pair of nonnegative, bounded functions $(w,z)$,
with $w,z\in W^{1,2;\,m}_{loc}(\mathbb{R}\times\mathbb{R})$ for each $m\in(1,\,\infty)$.

Moreover, since $u_\rho,$ $v_\rho\leq0$ by (\ref{monot}), we have 
$\partial_\theta W_j, \partial_\theta Z_j\le 0$ in $D$
 and therefore,  for each~$\sigma\in\mathbb{R}$,
\begin{equation}\label{33}
\hbox{$\mathbb{R}\ni\theta\mapsto w(\sigma,\theta)$ and $\mathbb{R}\ni\theta\mapsto z(\sigma,\theta)$
are nonincreasing.}
\end{equation}
 Since $w$ and $z$ are bounded and nonincreasing, we may define
\begin{eqnarray*}
  & &\phi(\sigma)=\underset{\theta\rightarrow+\infty}{\lim}w(\sigma,\,\theta),\quad\psi(\sigma)=\underset{\theta\rightarrow+\infty}{\lim}z(\sigma,\,\theta),
\end{eqnarray*}
which proves assertion (i).
\smallskip

(ii) We first observe that the properties of the sequence obtained in the previous paragraph
allow us to pass to the limit in the distribution sense in (\ref{13c}) and,
recalling $\partial_\theta W_j, \partial_\theta Z_j\le 0$ in $D$, it follows in particular that 
$(w,z)$ is a (continuous bounded) solution of
\begin{equation}\label{35a}
    \left\{
      \begin{array}{l}
        w_\sigma-\delta w_{\theta\theta}+\alpha w\ge c_1z^p, \\
         \noalign{\vskip 1mm}
        z_\sigma-z_{\theta\theta}+\beta z\ge c_1w^q,
              \end{array}
              \quad\hbox{in $\mathcal{D}'(\mathbb{R}^2)$.}
    \right.
\end{equation}

We can then obtain (\ref{27b}) by the following simple argument.
We check for instance the first inequality in (\ref{27b}), the other being completely similar.
Fix $\chi,\xi\in\mathcal{D}(\mathbb{R})$, with $\chi,\xi\ge 0$ and $\int_{\mathbb{R}}\chi=1$.
For $j\in\mathbb{N},$
replacing $\theta$ by $\theta+j$ in (\ref{35a}) and testing with $\xi(\sigma)\chi(\theta)$, we obtain
\begin{equation}\label{37c}
      \begin{array}{lll}
 \displaystyle\int_{\mathbb{R}} \int_{\mathbb{R}} \bigl[c_1 z^p-\alpha w\bigr](\sigma,\,\theta&{\hskip -3mm}+j)\xi(\sigma)\chi(\theta)\,d\theta d\sigma\\
 &= \Bigl\langle \bigl[c_1 z^p-\alpha w\bigr](\cdot,\,\cdot+j),\xi\otimes\chi\Bigr\rangle \\
                    \noalign{\vskip 1mm}
 &\le\ \Bigl\langle \bigl(w_\sigma-\delta w_{\theta\theta}\bigr)(\cdot,\,\cdot+j),\xi\otimes\chi\Bigr\rangle \\
                   \noalign{\vskip 1mm}
 &= \displaystyle \int_{\mathbb{R}} \int_{\mathbb{R}} (-\xi_\sigma(\sigma)\chi(\theta)-\delta \xi(\sigma)\chi_{\theta\theta}(\theta))w(\sigma,\,\theta+j)\,d\theta d\sigma.
         \end{array}
  \end{equation}
  Due to the boundedness of $w, z$, we may therefore apply the dominated convergence theorem on the first and last terms 
  of (\ref{37c}). Taking $\int_{\mathbb{R}}\chi=1$ and $\int_{\mathbb{R}}\chi_{\theta\theta}=0$ into account, we thus obtain 
$$
      \begin{array}{lll}
\displaystyle\int_{\mathbb{R}} \bigl[  c_1\psi^p-\alpha \phi\bigr](\sigma)\xi(\sigma)d\sigma
 &=&\displaystyle\int_{\mathbb{R}}\int_{\mathbb{R}} 
 \bigl[  c_1\psi^p-\alpha \phi\bigr](\sigma)\chi(\theta)\xi(\sigma)d\theta d\sigma \\
                    \noalign{\vskip 1mm}
 &\le&  \displaystyle \int_{\mathbb{R}}\int_{\mathbb{R}} \bigl(-\xi_\sigma(\sigma)\chi(\theta)
 -\delta \xi(\sigma)\chi_{\theta\theta}(\theta)\bigr)\phi(\sigma)d\theta d\sigma\\
                     \noalign{\vskip 1mm}
 &=& \displaystyle \int_{\mathbb{R}} -\xi_\sigma(\sigma)\phi(\sigma)d\sigma
        \end{array}
$$
and the conclusion follows.

\smallskip
(iii) Assume for contradiction that, for instance, $\phi(0)=0$ and $\psi(0)>0$. 
Then, by continuity, there exists $\eta>0$ such that
$[c_1\psi^p-\alpha \phi](\sigma)\ge \eta$ on $(-\eta,\eta)\subset I$.
Consequently $\phi'\ge \eta$ in $\mathcal{D}'(-\eta,\eta)$.
It is well known that this guarantees
$$
 \phi(y)-\phi(x)\ge \int_x^y \eta\, d\sigma=\eta(y-x)\quad\hbox{ for $-\eta<x<y<\eta$.}
$$
In particular $\phi(x)\le \phi(0)+\eta x=\eta x<0$ for all $x\in(-\eta,0)$: a contradiction.
\end{proof}


\section{Completion of proof of Proposition \ref{30}}

In this section, by using a contradiction argument and the results of Sections 3-5, we complete the proof of Proposition~\ref{30}.

\begin{proof}[Proof of Proposition \ref{30}]
 The upper estimates in (\ref{310})-(\ref{311}) follow from (\ref{28}) in Proposition \ref{65}.
To prove the lower estimates, since $u_\rho, v_\rho\le 0$ and since $u,v>0$ on $[T^\ast/2,\,T^\ast)\times [0,R)$ by the strong maximum principle,
 it suffices to show that, 
for each $\rho_1\in (0,\rho_0)$, 
$$\liminf_{t\to T^\ast}(T^\ast-t)^\alpha u(t,\rho_1)>0
\quad\hbox{and}\quad \liminf_{t\to T^\ast}(T^\ast-t)^\beta v(t,\rho_1)>0.$$
We argue by contradiction and assume for instance that
there exist  $\rho_1\in (0,\rho_0)$ and a sequence $t_j\to T^\ast$ such that
$$\lim_{j\to\infty}(T^\ast-t_j)^\alpha u(t_j,\rho_1)=0.$$
Set $\sigma_j:=-\log(T^\ast-t_j)\to\infty$, let $(W,Z)$ be defined by (\ref{defsimilvar1})-(\ref{defsimilvar2})
 and let $(\phi,\psi)$ be given by Proposition~\ref{30lem}(i).
Since $W(\sigma,\theta)\le W(\sigma,0)$ for all $\theta\in [0,(R-\rho_1)e^{\sigma/2}]$ due to (\ref{monot}), it follows from
(\ref{defphipsi}) that
$$\phi(0)=\lim_{\theta\to\infty}\Bigl(\lim_{j\to\infty}W(\sigma_j,\theta)\Bigr)
\le \lim_{j\to\infty}W(\sigma_j,0)=\lim_{j\to\infty}(T^\ast-t_j)^\alpha u(t_j,\rho_1)=0.$$
By Proposition~\ref{30lem}(ii) and (iii), it follows that $\psi(0)=\phi(0)=0$.
Therefore, with $\eta$ given by Proposition~\ref{15}, we deduce from (\ref{defphipsi}) that there exists $\theta_0>0$ such that
$$\lim_{j\to\infty}W(\sigma_j,\theta_0)\le \eta/2,\quad \lim_{j\to\infty}Z(\sigma_j,\theta_0)\le \eta/2.$$
Then, for all $j$ sufficiently large, we have
$$W(\sigma_j,\theta_0)\le \eta, \quad Z(\sigma_j,\theta_0)\le \eta$$
hence, in view of (\ref{defsimilvar1})-(\ref{defsimilvar2}),
$$(T^\ast-t_j)^\alpha u(t_j,\rho_1+\theta_0 \sqrt{T^\ast-t_j})\le \eta,\quad 
(T^\ast-t_j)^\beta v(t_j,\rho_1+\theta_0 \sqrt{T^\ast-t_j})\le \eta.$$
Taking $j$ large enough so that $\rho_1+\theta_0 \sqrt{T^\ast-t_j}<(\rho_0+\rho_1)/2$ and $T^\ast-t_j\le \tau_0$,
we conclude from Proposition~\ref{15} that $\rho_0$ is not a blow-up point: a contradiction.
\end{proof}

\section{Proof of Theorem~\ref{69} and verification of Examples~\ref{59bEx}.}
\setcounter{equation}{0}  

As a preliminary to the proof of Theorem~\ref{69}, we prove the following proposition.

\begin{pro}\label{69b}
Under the assumptions of Theorem~\ref{69}, there exists a constant $C > 0$ such that
\begin{eqnarray}\label{comparaison1}
\underset{Q_t}{\sup}\,u^{\frac{q+1}{p+1}} \leq C \,\underset{Q_t}{\sup}\, v,\quad T^\ast/2<t<T^\ast
\end{eqnarray}
and
\begin{eqnarray}\label{comparaison2}
\underset{Q_t}{\sup}\,v^{\frac{p+1}{q+1}} \leq C \,\underset{Q_t}{\sup}\, u,\quad T^\ast/2<t<T^\ast,
\end{eqnarray}
where $Q_t=(0,\,t)\times B(0,\,R).$
\end{pro}
\begin{proof}  As in \cite{tayachi}, we define the functions $U$, $V$ by:
\begin{eqnarray}
U(t)=\underset{Q_t}{\sup}\,u\quad\mbox{and}\quad\,V(t)=\underset{Q_t}{\sup}\, v.
\end{eqnarray}
Then $U$ and $V$ are positive continuous and nondecreasing on $(0,\,T^\ast)$. Also, since $(u,\,v)$ is a blowing-up solution, it follows that $U$ or $V$ diverges as $t\nearrow T^\ast.$
 We argue by contradiction. Assume that (\ref{comparaison1}) fails.
Then there exists a sequence $t_j\nearrow T^\ast$ as $j\rightarrow \infty$ such that $$V(t_j)U^{-\frac{q+1}{p+1}}(t_j)\rightarrow 0\quad\mbox{as}\quad j\rightarrow\infty.$$
It follows that $U$ must diverge as $t\nearrow T^\ast$. In the rest of the proof, we use the notation
\begin{eqnarray*}
\lambda_j:=U^{-\frac{1}{2\alpha}}(t_j)\underset{j\rightarrow \infty}{\rightarrow0},
\end{eqnarray*}
where $\alpha$ is given by (\ref{62}).

Let $(t'_j,\,x'_j)\in(0,\,t_j]\times B(0,\,R)$ be such that $u(t'_j,\,x'_j)\geq (1/2)U(t_j).$ 
We have $t'_j\rightarrow T^\ast$ as $j\rightarrow\infty.$
Now, we rescale the functions $U$ and $V$ by setting:
\begin{eqnarray*}
& &\phi_j(\sigma,\,y):=\lambda^{2\alpha}u(\lambda_j^2 \sigma+t'_j,\,\lambda_j y+x'_j),\\
& &\psi_j(\sigma,\,y):=\lambda^{2\beta}v(\lambda_j^2 \sigma+t'_j,\,\lambda_j y+x'_j),
\end{eqnarray*}
where $(\sigma,\,y)\in(-\lambda_j^{-2}t'_j,\,\lambda_j^{-2}(T^\ast-t'_j))\times(-\lambda_j^{-1}|x'_j|,\,\lambda_j^{-1}(R-|x'_j|))=:D_j$
and  $\alpha,\,\beta$ are given by (\ref{62}).
 Then, If we restrict $\sigma$ to $(-\lambda_j^{-2}t'_j,\,0],$ we obtain
\begin{eqnarray}\label{propriete}
0\leq\phi_j\leq1,\quad \phi_j(0,\,0)\geq 1/2\quad\mbox{and}\quad 0\leq\psi_j\leq V(t_j)U^{-\frac{q+1}{p+1}}(t_j)\rightarrow 0\;\mbox{as}\; j\rightarrow\infty.
\end{eqnarray}
On the other hand, $(\phi_j,\,\psi_j)$ solves the system:
\begin{equation*}
\left\{
  \begin{array}{llllll}
c_1\psi^p &\le& \phi_\sigma-\delta \Delta\phi &\leq& c_2\left(\psi^p+\phi^r+\lambda_j^{2 (\alpha+1)}\right),\\
c_1\phi^q &\le& \psi_\sigma- \Delta\psi &\leq& c_2\left(\phi^q+\psi^s+\lambda_j^{2 (\beta+1)}\right),
  \end{array}
\right.
\end{equation*}
on $D_j$. By using interior parabolic estimates, there exists a subsequence, still denoted by $(\phi_j,\,\psi_j)$,
converging uniformly on compact subsets of $(-\infty,\,0]\times\mathbb{R}^n$ to  $(\phi,\,\psi)$ a nonnegative 
(strong) solution  of
\begin{equation*}
\hspace{-0,2cm}\left\{
  \begin{array}{llll}
 \phi_\sigma-\delta \Delta\phi &\leq& c_2(\psi^p+\phi^r),\\
 \psi_\sigma- \Delta\psi &\geq& c_1 \phi^q.
  \end{array}
\right.
\end{equation*}
By (\ref{propriete}), it follows that $\phi(0,\,0)\geq 1/2$ and $\psi\equiv0.$ 
But the second equation implies $\phi\equiv0$:  a contradiction. This proves (\ref{comparaison1}). Statement (\ref{comparaison2}) follows by exchanging the roles of $u$, $p$, $r,$ $\alpha$
and $v$, $q$, $s,$ $\beta$.
\end{proof}

\begin{proof}[Proof of Theorem~\ref{69}]
Recall that, under the assumptions of
Theorem~\ref{69}, we know that $\|u(t)\|_\infty=u(t,\,0)$,
$\|v(t)\|_\infty=v(t,\,0)$ and $u(T^\ast,\,0)=v(T^\ast,\,0)=\infty$. By Proposition \ref{69b}, it follows that there
exists $C>0$ such that
\begin{eqnarray}\label{71}
& &v^p(t,\,0)\leq Cu^{r}(t,\,0).\\\label{72}
\mbox{and}\,\,\,& &u^q(t,\,0)\leq Cv^{s}(t,\,0).
\end{eqnarray}
Here and in the rest of the proof, $C$ denotes a positive constant which may vary from line to line.

On the other hand, since $v_t\geq0,$ $u_\rho\leq0$ and $v_\rho\leq0$ then,
\begin{eqnarray*}
 \frac{\partial}{\partial\rho}\left(\frac{1}{2}v_\rho^2+c_2v(u^q+v^s+1))\right)
& &=(v_{\rho\rho}+c_2(u^q+v^s+1))v_\rho+ c_2 q v u^{q-1}u_\rho+c_2 sv^s v_\rho\\
& &\leq(v_{\rho\rho}+F(u,\,v))v_\rho+ c_2 q v u^{q-1}u_\rho+c_2 sv^s v_\rho\\
& &=\left(v_t-\frac{n-1}{\rho}v_\rho\right)v_\rho+ c_2 q v u^{q-1}u_\rho+c_2 sv^s v_\rho\leq0.
\end{eqnarray*}
Consequently,
\begin{eqnarray*}
\left(\frac{1}{2}v_\rho^2+v F(u,\,v)\right)(t,\,\rho)\hspace{-0,6cm}& &\leq\left(\frac{1}{2}v_\rho^2+c_2v(u^q+v^s+1)\right)(t,\,\rho)\\
& &\leq\left(\frac{1}{2}v_\rho^2+c_2v(u^q+v^s+1)\right)(t,\,0)\\
& &\leq c_2 v (u^q+v^s+1)(t,\,0).
\end{eqnarray*}
 Moreover, by (\ref{71}), there exists $C>0$ such that
$v(u^q+v^s+1)(t,\,0)\leq C v^{s+1}(t,\,0)$, hence
\begin{eqnarray*}
\frac{1}{2}v_\rho^2(t,\,\rho)\hspace{-0,6cm}& &\leq C v^{s+1}(t,\,0),
\quad\hbox{for all $t\in(T^\ast/2,\,T^\ast)$ and $\rho\in [0,R]$}.
\end{eqnarray*}
Therefore,
\begin{eqnarray*}
\|v_\rho(t)\|_\infty\hspace{-0,6cm}& &\leq C v^{(s+1)/2}(t,\,0) = C v^{\frac{1}{2\beta}+1}(t,\,0),\quad\hbox{for all}\,\, t\in(T^\ast/2,\,T^\ast).
\end{eqnarray*}
 Arguing as in \cite[p. 187]{souplet},
we deduce that there exist $\eps_0, \eps_1>0$ such that
\begin{eqnarray*}
 \hspace{0,1cm}v(T^\ast,\,|x|)\geq \eps_0 |x|^{-2\beta},\quad\hbox{for all}\,\,|x|\in(0,\,\eps_1).
\end{eqnarray*}
The inequality on $G$ is obtained similarly.
\end{proof}

Finally, we verify the assertions made in Examples~\ref{59bEx}.
\smallskip

(i) Let $F, G$ be given by (\ref{defFGex1})-(\ref{defFGex2}).
Properties (\ref{classical})--(\ref{monotone}) are clear
(for $u,v>0$ in case some of the exponents belong to $(0,1)$).
To check (\ref{controle}), it suffices to estimate each of the products $u^{r_i} v^{s_i}$
with $r_i>0$ (the case $r_i=0$ being immediate). This follows from Young's inequality applied with 
the exponent $\frac{p(q+1)}{r_i(p+1)}>1$, writing
$$u^{r_i} v^{s_i}\le u^{\frac{p(q+1)}{p+1}}+v^{\frac{s_ip(q+1)}{ p(q+1)-r_i(p+1)}}
\le u^{\frac{p(q+1)}{p+1}}+C(v^p+1),$$
where we used $\frac{s_ip(q+1)}{ p(q+1)-r_i(p+1)}\le p$ due to (\ref{defFGex2}).
Property (\ref{controle1}) is obtained similarly.

It thus remains to verify (\ref{controle2}).
Fixing $C_2>C_1>0$, this amounts to finding $\mu,A,\kappa_1,\kappa_2>0$ with $\kappa_1\kappa_2<1$, such that 
\begin{eqnarray*}
& &R_1:= \lambda (\kappa_1 p-1-\mu)v^p+\displaystyle{\sum_{i=1}^{m}} \lambda_i\bigl(r_i+\kappa_1 s_i-1-\mu\bigr) u^{r_i} v^{s_i}\ge 0
\\
& &R_2:= \overline\lambda (\kappa_2 q-1-\mu)v^p+\displaystyle{\sum_{i=1}^{m}} \overline{\lambda}_i\bigl(\kappa_2 
\overline{r}_i+\overline{s}_i-1-\mu\bigr) u^{\overline{r}_i} v^{\overline{s}_i}\ge 0
\end{eqnarray*}
on the set $\{u,v\ge A\,|\, C_1\le \frac{u^{q+1}}{v^{p+1}}\le C_2\}.$
Fix $1/p<\kappa_1<1$, $1/q<\kappa_2<1$ and denote 
$$ I=\bigl\{i\in\{1,\dots,m\};\ r_i\textstyle\frac{p+1}{q+1}+s_i=p\bigr\},\quad
 \overline I=\bigl\{i\in\{1,\dots,m\};\ \overline{r}_i+\overline{s}_i\textstyle\frac{q+1}{p+1}=q\bigr\}.$$
Observe that if $i\in I$, then
$$r_i+\kappa_1 s_i-1\ge r_i+\frac{1}{p}\Bigl(p-r_i\frac{p+1}{q+1}\Bigr)-1=r_i\frac{pq-1}{p(q+1)}$$
and we may also assume $r_i>0$ (since otherwise $r_i=0$, $s_i=p$ and $\lambda_i u^{r_i} v^{s_i}$ can be included into the main term 
$\lambda v^p$). Similarly, if $i\in \overline I$, then
$$\kappa_2 \overline{r}_i+ \overline{s}_i-1\ge \frac{1}{q}\Bigl(q- \overline{s}_i\frac{q+1}{p+1}\Bigr)+\overline{s}_i-1
=\overline{s}_i\frac{pq-1}{q(p+1)}$$
and we may also assume $\overline{s}_i>0$.
Choosing 
$$0<\mu<\min\left(\frac{\kappa_1 p-1}{2},\ \frac{\kappa_2q-1}{2},\ 
\frac{pq-1}{p(q+1)}\min_{i\in I}r_i,\ \frac{pq-1}{q(p+1)}\min_{i\in \overline I}\overline{s}_i\right),$$
it follows that
\begin{eqnarray}\label{ineqR1} 
& &R_1\ge \mu v^p+\displaystyle{\sum_{i\in \{1,\dots,m\}\setminus I}
 \lambda_i\bigl(r_i+\kappa_1 s_i-1-\mu\bigr) u^{r_i} v^{s_i},}\\ \label{ineqR2} 
& &R_2\ge \mu u^q+\displaystyle{\sum_{i\in \{1,\dots,m\}\setminus \overline I}} \overline{\lambda}_i\bigl(\kappa_2 
\overline{r}_i+\overline{s}_i-1-\mu\bigr) u^{\overline{r}_i} v^{\overline{s}_i}.
\end{eqnarray}
Now consider $i\in  \{1,\dots,m\}\setminus I$. We have $r_i\textstyle\frac{p+1}{q+1}+s_i<p$ by (\ref{defFGex2}).
Therefore, on the set $D_A:=\{u,v\ge A\,|\, C_1\le \frac{u^{q+1}}{v^{p+1}}\le C_2\}$,
we have
$$u^{r_i} v^{s_i-p}\le C_2^{r_i/(q+1)}v^{-p+s_i+r_i(p+1)/(q+1)}\le C_2^{r_i/(q+1)}A^{-p+s_i+r_i(p+1)/(q+1)}\to 0$$
as $A\to\infty$. We get the similar property for $i\in  \{1,\dots,m\}\setminus \overline{I}$.
By (\ref{ineqR1})-(\ref{ineqR2}), we conclude that $R_1,R_2\ge 0$ on $D_A$ by taking $A$ large enough.
\smallskip

(ii) Let $F$, $G$ be given by (\ref{defFGex3})-(\ref{defFGex4}).
Properties (\ref{classical}) and (\ref{controle})--(\ref{controle1b}) are clear.
In order to verify (\ref{monotone}) and (\ref{controle2}), since $F_u=G_v=0$, it clearly suffices to find $\eta>0$ such that
$$vF_v(u,\,v)\ge (1+\eta)F(u,\,v),\ \ v\ge 0\quad\hbox{and}\quad uG_u(u,\,v)\ge (1+\eta)G(u,\,v),\ \ u\ge 0.$$
Setting $X=k\log(1+v)$, we compute
$$vF_v-(1+\eta)F=v^p\Bigl[(p-1-\eta)(1+\lambda \sin^2X)+2\lambda k\frac{v}{1+v}\cos X \sin X \Bigr].$$
Using 
$$|2\cos X\sin X|\le \frac{\cos^2X}{\sqrt{1+\lambda}}+\sqrt{1+\lambda}\sin^2X
= \frac{1+\lambda \sin^2X}{\sqrt{1+\lambda}},$$
we get
$$vF_v-(1+\eta)F\ge v^p\Bigl[p-1-\eta-\frac{\lambda k}{\sqrt{1+\lambda}}\Bigr](1+\lambda \sin^2X)\ge 0,\quad v\ge 0,$$
under assumption (\ref{defFGex4}) if we choose $\eta>0$ small.
The inequality for $G$ is similar.

\bibliographystyle{plain}

\end{document}